\documentclass[11pt,reqno]{amsart}
\usepackage[T1]{fontenc}
\usepackage{lmodern}
\usepackage{microtype}
\usepackage[a4paper,margin=29mm,headheight=14pt]{geometry}
\usepackage{amsmath,amssymb,amsthm,mathtools,booktabs,array}
\usepackage[hidelinks,bookmarksnumbered,pdfencoding=auto]{hyperref}
\usepackage{enumitem}
\setlist[enumerate]{label=\textup{(\roman*)},leftmargin=2.2em,itemsep=3pt,topsep=5pt}
\setlist[itemize]{leftmargin=2em,itemsep=2pt,topsep=4pt}
\numberwithin{equation}{section}
\newtheorem{theorem}{Theorem}[section]
\newtheorem{lemma}[theorem]{Lemma}
\newtheorem{proposition}[theorem]{Proposition}
\newtheorem{corollary}[theorem]{Corollary}
\theoremstyle{definition}
\newtheorem{definition}[theorem]{Definition}
\newtheorem{question}[theorem]{Question}
\theoremstyle{remark}
\newtheorem{remark}[theorem]{Remark}
\newtheorem{example}[theorem]{Example}
\newcommand{\Zm}{\mathbb Z_m}
\newcommand{\Z}{\mathbb Z}
\newcommand{\one}{\mathbf 1}
\newcommand{\id}{\operatorname{id}}
\newcommand{\cyc}{\operatorname{cyc}}
\newcommand{\Ret}{\operatorname{Ret}}
\newcommand{\sgn}{\operatorname{sgn}}
\newcommand{\Sym}{\operatorname{Sym}}
\newcommand{\Cay}{\operatorname{Cay}}
\newcommand{\cA}{\mathcal A}
\newcommand{\cB}{\mathcal B}
\newcommand{\cO}{\Omega}
\newcommand{\ep}{\varepsilon}
\newcommand{\tr}{\mathsf T}
\newcommand{\arxivref}[1]{\href{https://arxiv.org/abs/#1}{arXiv:#1}}
\newcommand{\doiref}[1]{\href{https://doi.org/#1}{doi:#1}}
\allowdisplaybreaks[2]
\title[Hamilton decompositions of equal-side directed tori]{Hamilton decompositions of equal-side directed tori}
\author{SangHyun Park}
\date{September 17, 2026}
\hypersetup{pdftitle={Hamilton decompositions of equal-side directed tori},pdfauthor={SangHyun Park},pdfsubject={All moduli, integer incidence selection, and marked cyclic lifts}}

\begin{document}
\begin{abstract}
Let $D_d(m)$ be the Cartesian product of $d$ positively oriented directed
cycles of length $m$. We prove that $D_d(m)$ decomposes into $d$ directed
Hamilton cycles for every $m\ge3$ and $d\ge2$. The construction proceeds
by splitting coordinate directions in directed multitori. An integer
selection theorem supplies unit voltages compatible with the prescribed
arc multiplicities; at even modulus, the decisive condition is the
parity of each incidence component. For even $m$ and odd $d\ge7$,
we satisfy this condition with a factorization having one additional
cycle. A relative lifting theorem preserves the first-return data of a
three-colour recolouring through successive coordinate splits, after
which the recolouring removes the additional cycle on a set whose size
is independent of the dimension. Explicit constructions complete the
low-dimensional cases. The theorem yields Hamilton decompositions of
Cartesian products of equal-order, equal-degree Hamilton-decomposable
digraphs, of abelian Cayley digraphs whose generators partition into
module bases, and of a family of height-stretched directed tori.
\end{abstract}

\maketitle

\section{Introduction}\label{sec:introduction}
For integers $m\ge2$ and $d\ge2$, let $\Zm=\Z/m\Z$ and write
\[
 D_d(m)=\Cay\bigl(\Zm^d,\{e_1,\ldots,e_d\}\bigr)
       =\vec C_m^{\,\square d},
\]
where $e_1,\ldots,e_d$ are the standard coordinate vectors. We call
$D_d(m)$ the \emph{positive equal-side directed torus}. A Hamilton
decomposition is a partition of its arc set into directed spanning
cycles. Since every vertex has indegree and outdegree $d$, such a
partition has exactly $d$ members. The coordinate directions give a
natural factorization, each factor consisting of $m^{d-1}$ cycles of
length $m$. The problem is to replace these factors by spanning cycles
using the same positive coordinate arcs.

\begin{theorem}\label{thm:main}
For every $m\ge3$ and every $d\ge2$, the digraph $D_d(m)$ has a decomposition
into $d$ directed Hamilton cycles.
\end{theorem}

The theorem holds in every dimension at each fixed cycle length
$m\ge3$, including composite lengths. Forgetting the arc orientations
gives a Hamilton decomposition of $C_m^{\square d}$ in which every cycle
respects the prescribed positive orientation. The lower bound on $m$ is sharp for a theorem covering every
dimension: $D_3(2)$ has no Hamilton decomposition~\cite{AS82binary}, whereas
the square $D_2(2)$ is included in Proposition~\ref{prop:square}.
The twelve directed Hamilton cycles of $D_3(2)$ contain no
arc-disjoint triple.

\subsection{Cartesian products and changes of basis}
The modulus becomes the common order of the factors when their Hamilton
cycles are combined in a Cartesian product.

\begin{corollary}[Equal-order Cartesian products]\label{cor:cartesian-products}
Let $d\ge2$, $n\ge3$, and $r\ge1$. Suppose each finite loopless digraph
$G_1,\ldots,G_d$ has $n$ vertices and a decomposition into $r$ directed
Hamilton cycles. Then $G_1\square\cdots\square G_d$ has a decomposition
into $dr$ directed Hamilton cycles. In particular, every Cartesian power
of a Hamilton-decomposable digraph on at least three vertices is
Hamilton-decomposable.
\end{corollary}

For fixed $n$ and $d$, this assertion for all $r$ and all such factors is
equivalent to the single torus instance $D_d(n)$
(Proposition~\ref{prop:universal-transfer}). Indeed, choosing the same
Hamilton-factor index in each coordinate partitions the product into
spanning copies of $D_d(n)$. The all-moduli theorem consequently gives a
product theorem at every common order $n\ge3$. Equality of orders is a
substantive hypothesis in the directed setting: $\vec C_3\square\vec C_4$
has no Hamilton cycle.

An automorphism of the module $(\Z/m\Z)^d$ transports the standard
coordinate directions to any module basis. A partition of a connection
multiset into $k$ bases therefore gives $kd$ Hamilton factors
(Corollary~\ref{cor:basis-partition}). At a prime power, Edmonds' matroid
partition theorem yields a criterion in terms of ranks over the residue
field, including the moduli $2^a$ for $a\ge2$. For a general modulus the
same partition must work at all prime divisors, a compatibility condition
illustrated by Example~\ref{ex:crt-bases}. A further construction preserves
the return permutations while increasing the number of height layers;
it gives the family on $\Z_{am}\times\Z_m^{d-1}$ in
Corollary~\ref{cor:skew-height}.

\subsection{Previous work}
For the classical theory of Hamilton cycles in Cayley graphs and
digraphs, see Curran and Gallian~\cite{CG96}. In products of directed
cycles, Trotter and Erd\H{o}s~\cite{TE78} characterized Hamiltonicity
in two factors, and Curran and Witte~\cite{CW85} proved that every
product of three or more nontrivial directed cycles is Hamiltonian. A full
Hamilton decomposition imposes a stronger condition already in two
factors. Keating's criterion~\cite{Keating85}, recalled together with the
Hamiltonicity criterion in \cite[Theorems~2.1 and~2.3]{DMM22}, gives the
example $\vec C_3\square\vec C_6$: it has a Hamilton cycle but no Hamilton
decomposition. We return to these arithmetic distinctions in
Example~\ref{ex:two-factor-boundary}.

For undirected cycle products, Hamilton decomposition is classical
\cite{AS82}; see also \cite[Theorem~2.5]{DMM22}. Stong developed related
theorems for Cartesian products of undirected graphs~\cite{Stong91} and
for symmetric directed products~\cite{Stong06}, where each edge is
replaced by its two orientations. Theorem~\ref{thm:main} keeps just the
positive coordinate arcs. Bogdanowicz~\cite{Bog17,Bog20} studied
equal-length cycle decompositions and the construction of individual
Hamilton cycles, while Meng and Huang~\cite{MH97} established sufficient
conditions for Hamiltonicity and Hamilton decomposition in finite
abelian Cayley digraphs, with necessity in the $2$-regular case.

For arc-disjoint Hamilton paths, Darijani,
Miraftab and Witte Morris~\cite{DMM22} constructed two such paths in
products of two directed cycles, in products of at least four directed
cycles, and in many three-factor cases. The remaining two-generated
abelian and three-factor path cases are treated in~\cite{ParkPaths26}.
These results permit unequal lengths; the equal-side hypothesis here
allows all arcs to be partitioned into closed spanning factors.
Three-dimensional cycle constructions are described in~\cite{Park3}
and in Knuth's account~\cite{Knuth26}, with the associated work of
Aquino-Michaels archived in~\cite{Aquino26}. Our use of dimension three
requires a decomposition agreeing with a specified factorization at
four vertices, so that its recolouring can be preserved under lifting.

An earlier version of this paper~\cite{ParkVtwo} established the
odd-modulus theorem. The second proof~\cite{ParkShort} introduced
repeated incidence selection and coordinate splitting; Section~\ref{sec:odd}
strengthens its selection argument to a modulus-independent integer
theorem. The present paper supersedes~\cite{ParkVtwo} by extending the
result to every $m\ge3$. The even-modulus proof combines componentwise
incidence parity with a relative lift that preserves the effect of a
recolouring. Its three-dimensional construction in
Section~\ref{sec:anchored} proves the existence result of~\cite{Park3}
with four prescribed agreements needed for iteration.

The graphs have order $m^d$ and degree $d$, so their sparsity persists
as the cycle length increases. The construction complements the
robust-expansion theory of K\"uhn and Osthus~\cite{KO13}, whose
Hamilton decomposition theorem applies to sufficiently large regular
robust outexpanders of degree linear in their order.

\subsection{Coordinate splitting and incidence parity}
The common operation acts on a directed multitorus
$T_m(a_1,\ldots,a_n)$, with $a_i$ labelled copies of direction $i$ and
$\sum_i a_i=d$. Splitting one multiplicity $a\ge2$ into
$\lceil a/2\rceil$ and $\lfloor a/2\rfloor$ raises the number of coordinates
by one while retaining the $d$ colours. In suitable coordinates the
successor of colour $c$ becomes
\[
 \widehat S_c(x,t)=(S_cx,t+\delta_c(x)),\qquad \delta_c(x)\in\{0,1\}.
\]
The number of selected colours at each source determines the two new
arc multiplicities. On an old cycle $C$, the total voltage
$\sum_{x\in C}\delta_c(x)$ determines the number of lifted cycles; a unit
modulo $m$ gives one cycle on $C\times\Zm$. The latter principle is the
usual cyclic-voltage argument, closely related to the Factor Group
Lemma~\cite[Lemma~9.11]{CW85}. The selection problem is to obtain unit
totals for every colour-cycle while respecting the multiplicities at
every source.

For odd $m$, this selection has a short integer solution. On an incidence
system with column set $\cO$ and supports $A_B$, choose three subsets
$J_B^0,J_B^1,J_B^*$ of size $\lfloor|A_B|/2\rfloor$ and put
\[
 D_\lambda=\sum_B\bigl(1_{\lambda\in J_B^0}
                 +1_{\lambda\in J_B^1}-2\,1_{\lambda\in J_B^*}\bigr).
\]
Theorem~\ref{thm:integer-selector} gives
$D_\lambda\in\{-2,-1,1,2\}$ for every column. Using the first two subsets
on two rows and the third on the remaining $m-2$ rows makes each column
total congruent to $D_\lambda$ modulo $m$. All four values are units when
$m$ is odd, and $d-1$ successive splits of $T_m(d)$ give $D_d(m)$.

At even modulus every unit is odd. Since each block contributes an even
number of selections, each incidence component must contain an even
number of cycle-columns. Theorem~\ref{thm:radius-dichotomy} identifies the
same boundary over the integers: deviations in $\{-1,1\}$ are possible
exactly under this parity condition. Proposition~\ref{prop:pinned-selection} supplies
these deviations together with the prescribed exclusions and
intersection connectivity that preserve the condition through further
splits. For even $d$, the initial one-block factorization already has an
even number of cycle-columns, and this completes the construction.

\subsection{First returns and the additional cycle}
When $m$ is even and $d$ is odd, parity also enters through permutation
signs. A valid two-colour exchange preserves the product of the two
successor signs, whereas the canonical coordinate factorization and a
Hamilton decomposition have opposite sign products. If
\[
 E=\sum_{c=1}^d\bigl(\cyc(S_c)-1\bigr)
\]
is the cycle excess, then excess one is the least value permitted in
the canonical sign class (Proposition~\ref{prop:cycle-excess}).
We construct such a factorization for odd $d\ge7$: one colour has two
cycles, and every other colour is Hamilton. Its $d+1$ cycle-columns
satisfy the incidence parity conditions. A three-colour recolouring
removes the additional cycle after all coordinate multiplicities have
been split.

To preserve this recolouring, we retain the first-return permutations
on its support $U$. For a successor $S$ and a permutation $r$ supported
on $U$, source reconnection gives
\[
 \cyc(Sr)=h_U(S)+\cyc\bigl(\Ret_U(S)r\bigr),
\]
where $h_U(S)$ counts the cycles disjoint from $U$. The relative lift
concentrates its nonzero voltages in one open interval between
consecutive marked vertices. It lengthens that interval while preserving
the first-return permutation, and hence preserves the effect of the
recolouring. The enlarged interval contains every new point off the
zero section, providing the positions required by the next split.
Thus the same marked construction passes through all later dimensions;
its final recolouring has
\[
 |U_m|=m^3-m^2-4m+7
\]
sources, independently of $d$ (Corollary~\ref{cor:support}).
Dimensions three and five are supplied by explicit constructions.

\par\medskip
\noindent\begin{minipage}{\textwidth}
The cases of Theorem~\ref{thm:main} are organized as follows.
\begin{center}
\small
\renewcommand{\arraystretch}{1.12}
\begin{tabular}{@{}p{0.27\textwidth}p{0.43\textwidth}p{0.24\textwidth}@{}}
\toprule
Parameters & Construction & Conclusion\\
\midrule
Odd $m\ge3$, $d\ge2$ & Integer selection and coordinate splitting &
Theorem~\ref{thm:oddmod}\\[3pt]
Even $m\ge4$, even $d\ge2$ & Incidence parity and coordinate splitting &
Theorem~\ref{thm:even-dim}\\[3pt]
Even $m\ge4$, $d=3$ & Anchored cyclic-star construction &
Proposition~\ref{prop:anchor}\\[3pt]
Even $m\ge4$, $d=5$ & Transversal cycle joining &
Theorem~\ref{thm:d5}\\[3pt]
Even $m\ge4$, odd $d\ge7$ & Relative splitting and recolouring &
Theorem~\ref{thm:oddconstruction}\\
\bottomrule
\end{tabular}
\end{center}
\end{minipage}
\par\medskip
Sections~\ref{sec:preliminaries} and~\ref{sec:relative} provide the lifting
principles. Sections~\ref{sec:consequences}--\ref{sec:extensions} develop
their consequences and further questions; the appendices contain the
return calculations, selection refinements, and formalization record.

\section{Permutation factorizations and coordinate splitting}\label{sec:preliminaries}
All permutations act on the left, and products act from right to left. For a
permutation $S$ of a finite set, $\cyc(S)$ denotes the number of cycles in its
complete cycle decomposition, including fixed points.

\subsection{Directed multitori}
For positive integers $a_1,\ldots,a_n$, let
\[
 T_m(a_1,\ldots,a_n)
\]
be the directed multigraph on $\Zm^n$ having $a_i$ separately labelled arcs from
$x$ to $x+e_i$ at each vertex. Put $d=\sum_i a_i$. A permutation factorization is a
list of permutations $S_1,\ldots,S_d$ and direction functions
$\partial_c:\Zm^n\to\{1,\ldots,n\}$ such that
\begin{equation}\label{eq:direction-quota}
 S_c(x)=x+e_{\partial_c(x)},\qquad
 \bigl|\{c:\partial_c(x)=i\}\bigr|=a_i
 \quad\text{for all }x,i.
\end{equation}
Assigning distinct labels to the parallel arcs identifies
\eqref{eq:direction-quota} with an arc partition. The multiplicity
condition supplies the outgoing arcs of each direction, and the
bijectivity of each $S_c$ supplies exactly one incoming arc of colour
$c$ at every vertex.

A colour--cycle label is a pair $(c,C)$, where $C$ is one cycle of $S_c$. We use
\begin{equation}\label{eq:orbit-set}
 \cO=\{(c,C):C\in\operatorname{Cyc}(S_c),\ 1\le c\le d\}
\end{equation}
for the full label set. A partition into blocks of size $m$ is
\emph{cycle-consistent} if all direction functions and all colour--cycle
labels are constant on each block. These labels refer to the factorization before lifting; a recolouring
acts through a permutation of its source set.

\subsection{Return to a height layer}
The return map on a single height layer determines the cycles of a
layered permutation. We use the following form of the criterion in
\cite[Theorem~3.2]{ParkVtwo}.

\begin{lemma}[Layer return]\label{lem:layer-return}
Let $X=\dot\bigcup_{t\in\Zm}X_t$ be finite, and suppose
$S:X\to X$ maps each $X_t$ bijectively onto $X_{t+1}$.
Then $S$ is a permutation. Its cycles correspond bijectively to the cycles
of $S^m|_{X_0}$, and their lengths are multiplied by $m$.
Consequently $S$ is one cycle if and only if $S^m|_{X_0}$ is one cycle.
\end{lemma}
\begin{proof}
The layer bijections give a bijection of $X$. Every orbit visits the layers
in their cyclic order and returns to $X_0$ exactly every $m$ steps.
Intersection with $X_0$ and expansion along those $m$ steps give the stated
bijection of cycles and the length formula.
\end{proof}

For a multitorus, take $X_t=\{x:x_1+\cdots+x_n=t\}$. Every positive
coordinate step increases the height by one, so a Hamilton decomposition
is obtained by satisfying the direction multiplicities at every source,
choosing bijective layer maps, and making each colour return cyclic.
This is the form used in the low-dimensional constructions.

\subsection{Cyclic lifts}
Let $S$ be a permutation of $X$ and let $\delta:X\to\Zm$. Define
\[
 \widehat S(x,t)=(Sx,t+\delta(x))\qquad(x\in X,\ t\in\Zm).
\]
The total voltage on a cycle $C$ of $S$ is
$K_C=\sum_{x\in C}\delta(x)\in\Zm$.

The number and lengths of the lifted cycles are determined by the
order of $K_C$ in $\Zm$; the unit case appears in
\cite[Lemma~8.1]{ParkVtwo}.

\begin{lemma}[Cycle lifting]\label{lem:lift}
The permutation $\widehat S$ has, over a cycle $C$ of length $L$, exactly
$\gcd(m,K_C)$ cycles, each of length $mL/\gcd(m,K_C)$. In particular, a unit total
voltage lifts $C$ to a single cycle on $C\times\Zm$.
\end{lemma}
\begin{proof}
The inverse is $(y,s)\mapsto(S^{-1}y,s-\delta(S^{-1}y))$. After one traversal of
$C$, the fibre coordinate is translated by $K_C$. Its orbits have cardinality
$m/\gcd(m,K_C)$, giving the assertion.
\end{proof}

\begin{lemma}[Coordinate splitting]\label{lem:physical}
Suppose $S_c$ factorize $T_m(a_1,\ldots,a_n)$, and fix a direction $i$ and an
integer $b$ with $0<b<a_i$. Choose $\delta_c(x)\in\{0,1\}$ such that
\begin{equation}\label{eq:physical-delta}
 \delta_c(x)=0\ \text{if }\partial_c(x)\ne i,
 \qquad \sum_{c=1}^d\delta_c(x)=b.
\end{equation}
Then the lifts $\widehat S_c(x,t)=(S_cx,t+\delta_c(x))$ factorize
\[
 T_m(a_1,\ldots,a_{i-1},a_i-b,b,a_{i+1},\ldots,a_n).
\]
\end{lemma}
\begin{proof}
Identify the coordinates after splitting by
\begin{equation}\label{eq:sum-sheet}
 (x,t)\longmapsto
 (x_1,\ldots,x_{i-1},x_i-t,t,x_{i+1},\ldots,x_n).
\end{equation}
An old $i$-arc with $\delta_c=0$ increases the first child coordinate; one with
$\delta_c=1$ increases the second. Every other old arc remains a positive
coordinate arc. Equation~\eqref{eq:physical-delta} gives the two new
multiplicities, and Lemma~\ref{lem:lift} gives bijectivity. The lifted factors
therefore use the same colours and partition the new positive coordinate arcs.
\end{proof}

\subsection{The square torus}
A pair of opposite unit returns gives the elementary square construction
\cite[Theorem~4.1]{ParkVtwo}.

\begin{proposition}\label{prop:square}
For every $m\ge2$, $D_2(m)$ has a Hamilton decomposition.
\end{proposition}
\begin{proof}
Put $s=x+y$. Let $F_0(x,y)$ use direction $e_1$ when $s=0$ and $e_2$
otherwise, and let $F_1(x,y)$ use the other direction.
At every source the two directions are used exactly once.
In coordinates $(s,x)$ the maps are
\[
 F_0(s,x)=(s+1,x+\one_{\{s=0\}}),\qquad
 F_1(s,x)=(s+1,x+1-\one_{\{s=0\}}).
\]
They are bijective between successive layers. After $m$ steps their
returns translate $x$ by $1$ and $-1$, respectively, both units modulo
$m$. Lemma~\ref{lem:layer-return} makes each a cycle of length $m^2$.
\end{proof}

\section{A short proof for odd moduli}\label{sec:odd}\label{sec:selection}
The coordinate split requires the following incidence selection.
Let $\cO$ be a finite column set. For each member $B$ of an indexed family
$\cB$, let $A_B\subseteq\cO$ have size $a_B\ge2$. There are $m$ identical copies,
or rows, of $A_B$, indexed by $t\in\Zm$. Set $b_B=\lfloor a_B/2\rfloor$.
The incidence graph has vertex classes $\cB$ and $\cO$, with an edge $B\lambda$
when $\lambda\in A_B$. All columns are included; a column absent from every
support is an isolated vertex.

\subsection{A three-row selection over the integers}
An incidence-tree construction, following~\cite{ParkShort}, confines
all column deviations to a fixed set of nonzero integers. Put
\[
 \mathcal P=\{-2,-1,1,2\}.
\]
For three subsets $J_B^0,J_B^1,J_B^*$ of $A_B$, each of size $b_B$, define
\begin{equation}\label{eq:integer-charge}
 D_\lambda=\sum_{B\in\cB}
 \bigl(\one_{\{\lambda\in J_B^0\}}+
       \one_{\{\lambda\in J_B^1\}}-2\one_{\{\lambda\in J_B^*\}}\bigr).
\end{equation}
We regard each $D_\lambda$ as an integer.

\begin{theorem}[Three-row integer selection]\label{thm:integer-selector}
Suppose that every column belongs to at least one support. There are
three subsets $J_B^0,J_B^1,J_B^*\subseteq A_B$ of size
$\lfloor |A_B|/2\rfloor$ for every block $B$ such that
$D_\lambda\in\mathcal P$ for every column $\lambda$.
The bound $|D_\lambda|\le2$ is best possible if zero charges are forbidden.
\end{theorem}
\begin{proof}
We first give three local constructions on a support $A$ of size $a$,
with row quota $b=\lfloor a/2\rfloor$.

A constant choice in all three rows has charge zero. Next, for any
$T\subseteq A$ of size at least two, pair its elements, leaving one triple
when $|T|$ is odd. On a pair $(p,q)$, choose $p$ in row $0$ and $q$ in
rows $1,*$, giving charges $(1,-1)$. On a triple $(p,q,r)$, choose one
element in each of rows $0,1,*$, giving $(1,1,-2)$. Each row uses
$\lfloor |T|/2\rfloor$ elements of $T$. Fill each row with the same
$b-\lfloor |T|/2\rfloor$ elements of $A\setminus T$; this is possible since
\[
 (a-|T|)-\bigl(b-\lfloor |T|/2\rfloor\bigr)
 =\lceil a/2\rceil-\lceil |T|/2\rceil\ge0.
\]
Thus all columns of $T$ have charges in $\mathcal P$ and all other
columns have charge zero.

Finally, for distinct $p,q\in A$ and $u\in\mathcal P$, we can realize
$u(\varepsilon_q-\varepsilon_p)$. For $u=1$, choose $(q,p,p)$ in rows
$(0,1,*)$; for $u=2$, choose $(q,q,p)$. Reverse $p,q$ for negative $u$,
and fill every row with the same $b-1$ other columns.

Work in one incidence component. Choose a minimal connected subgraph
containing all its column vertices. It is a tree, and no block vertex is
a leaf. Root it at a column $\rho$. Every block $B$ has a parent column
$p_B$ and a nonempty set $C_B$ of child columns. Choose one block $B_0$
adjacent to $\rho$. At $B_0$, apply the construction with
$T=\{\rho\}\cup C_{B_0}$; at every other block with at least two
children, apply it with $T=C_B$. Use the zero construction outside the tree.
Leave the remaining blocks, each with one child column, unassigned.

Process the column vertices from the root outwards. At a column $y$,
its parent block has supplied a charge $r\in\mathcal P$; at the root this
is supplied by $B_0$. All other already assigned child blocks contribute
zero at $y$. If $t$ child blocks of $y$ remain unassigned, each with
one child column, choose
$u_1,\ldots,u_t,v\in\mathcal P$ with
\[
 r=u_1+\cdots+u_t+v.
\]
Each element of $\mathcal P$ is a sum of two elements of $\mathcal P$, by
\begin{equation}\label{eq:palette-splitting}
 1=2+(-1),\qquad -1=(-2)+1,\qquad
 2=1+1,\qquad -2=(-1)+(-1).
\end{equation}
At the block with sole child $z_i$, use the transfer
$u_i(\varepsilon_{z_i}-\varepsilon_y)$. The final charge at $y$ is $v$,
and each $z_i$ receives $u_i\in\mathcal P$, so the construction
continues by induction down the tree. Entries outside the indicated
tree columns contribute zero. Thus the additional edges of the original
incidence graph leave every final deviation in $\mathcal P$.

For sharpness, take a single three-column block. Its three integer
charges sum to zero, because all three row quotas agree. Three values
from $\{-1,1\}$ cannot sum to zero.
\end{proof}

\begin{corollary}[Four-valued odd-modulus selection]\label{cor:odd-selector}
For every odd $m\ge3$, choose $J_{B,0}=J_B^0$, $J_{B,1}=J_B^1$, and
$J_{B,t}=J_B^*$ for $2\le t<m$. Every row has quota $b_B$, and every column
total is a unit modulo $m$, in fact one of $\pm1,\pm2$ modulo $m$.
\end{corollary}
\begin{proof}
The total at column $\lambda$ is
\begin{equation}\label{eq:integer-to-modular}
 \sum_{B,t}\one_{\{\lambda\in J_{B,t}\}}
 =m\sum_B\one_{\{\lambda\in J_B^*\}}+D_\lambda.
\end{equation}
Since both $1$ and $2$ are units at odd modulus, the integer
construction applies to every odd $m$.
\end{proof}

\begin{remark}[Row sizes and residue constraints]\label{rem:non-saturation}
The selection theorem finds unit column totals subject to fixed row
sizes. For example, three quota-one rows on a three-column support have
nonnegative column counts summing to three. The unit profile $(1,1,1)$
is realized by choosing a different column on each row, whereas the
zero-sum residue profile $(2,2,2)$ modulo three would require at least
six selections. The row sizes thus distinguish realizable unit profiles
from the larger set allowed by the total congruence.
\end{remark}

\subsection{Repeated coordinate splitting}\label{sec:direct}
In the one-coordinate multitorus $T_m(d)$, each colour uses a labelled
copy of the positive $m$-cycle, and the whole vertex set forms one
cycle-consistent block. Each split adds one coordinate and decreases
$\sum_i(a_i-1)$ by one.

\begin{theorem}[Odd moduli]\label{thm:oddmod}
For every odd $m\ge3$ and every $d\ge2$, $D_d(m)$ has a Hamilton decomposition.
\end{theorem}
\begin{proof}
More generally, suppose a multitorus has a Hamilton factorization with
cycle-consistent blocks of size $m$. Fix a direction $i$ of multiplicity
$a\ge2$. In each block, take the support of colours using direction $i$.
Every colour occurs in such a support: a Hamilton cycle cannot avoid a
coordinate direction, since it would keep that coordinate constant.
Corollary~\ref{cor:odd-selector} selects $\lfloor a/2\rfloor$ colours at
every source and gives every colour unit total voltage. Use its binary
indicators as $\delta_c$ in Lemma~\ref{lem:physical}. The split is a
factorization, and Lemma~\ref{lem:lift} makes every lifted factor Hamilton.

The new fibres over individual old sources are blocks of size $m$.
Their direction assignments are independent of the new fibre label;
they lie in a single cycle of every lifted colour. Thus the same
hypothesis holds after the split. Starting from $T_m(d)$, repeat exactly
$d-1$ times to obtain $T_m(1,\ldots,1)=D_d(m)$.
\end{proof}

Repeated splitting therefore gives every odd-modulus decomposition
from the one-coordinate factorization. At even modulus, the same
operation leads to the incidence parity condition considered next.

\section{Incidence parity at even modulus}\label{sec:even-selection}
Retain the incidence system and half quotas of Section~\ref{sec:odd}.
At even modulus, unit column totals are odd, whereas each block contributes
an even total. The condition must therefore be checked separately in
each incidence component, with all columns included.

A family of nonempty subsets is \emph{intersection-connected} if its
intersection graph is connected. We require this property separately within each
block, for the selected row supports and for their complements, whenever the
corresponding sets have size at least two.

\subsection{Selection with exclusions and connectivity}

\begin{proposition}[Connected selection with prescribed exclusions]\label{prop:pinned-selection}
Let $m\ge4$ be even, and let $\cO_\mathrm{dist}\subseteq\cO$ be a set of
distinguished columns satisfying
$|A_B\cap\cO_\mathrm{dist}|\le1$ for every $B$. The following are equivalent.
\begin{enumerate}
\item Each component of the incidence graph contains an even number of column
vertices.
\item One can choose $J_{B,t}\subseteq A_B$, with $|J_{B,t}|=b_B$, so that every
column total is $1$ or $-1$ modulo $m$, no distinguished column belongs to $J_{B,0}$,
and, for each $B$, both families
$\{J_{B,t}:t\in\Zm\}$ and $\{A_B\setminus J_{B,t}:t\in\Zm\}$ are
intersection-connected whenever their members have size at least two.
\end{enumerate}
\end{proposition}
\begin{proof}
For necessity, sum the selected entries in an incidence component. Each block
contributes $mb_B$, an even number, while every column contributes an odd number.

For sufficiency, we first choose and orient pairs of columns, then realize
each oriented pair by a row selection. The choice of exceptional rows will
provide the two connectivity conclusions.

\emph{Parity and orientation.}
Take a spanning tree in each component of the incidence graph. Select its
edges so that column vertices have odd selected degree and block vertices
have even selected degree. To construct this selection, remove leaves
successively, selecting a leaf's parent edge exactly when its remaining
parity requirement is odd, and updating the requirement at its parent.
The last requirement is zero because the component has an even number of
column vertices.
Pair the selected incidences at each block. Since each column occurs at most
once in that block, these are disjoint local pairs. Across all blocks, they
define a multigraph on $\cO$ in which every vertex has odd degree; retain the
block label on each edge.

For each connected component of this multigraph, introduce one new vertex
and join it once to every vertex of that component. The component has an
even number of vertices by the handshaking lemma, so the new vertex and all
old vertices have even degree. Orient an Euler circuit of the enlarged
component and remove the new edges. At every column, the remaining indegree
minus outdegree is $+1$ or $-1$. This orients all local pairs while retaining
their block labels.

\emph{Row quotas and column totals.}
In a block with $a$ columns, write $s$ for the number of local pairs and put
$b=\lfloor a/2\rfloor$. For a directed pair $u\to v$, choose an exceptional
row $\eta$: select $v$ on row $\eta$ and $u$ on each of the other $m-1$ rows.
The pair contributes exactly one selected entry to every row and contributes
$-1,+1$ modulo $m$ at $u,v$, respectively. Thus the column totals of all
pairs equal the oriented divergences modulo $m$.
Choose a further $b-s$ unpaired columns on every row. If the distinguished
column is unpaired, exclude it from this fixed choice. There are enough
columns: for $a=2b$, an unpaired distinguished column gives $s\le b-1$ and
\[
 a-2s-1\ge b-s;
\]
for $a=2b+1$, the same inequality follows from $s\le b$. A paired
distinguished column is already outside the unpaired set. The resulting
rows have size $b$, and every fixed selected column contributes zero modulo
$m$.

\emph{Exclusions and connectivity.}
If the distinguished column is the first endpoint of a pair, put that
pair's exceptional row at $0$; if it is the second endpoint, put the
exceptional row at $1$. When $s\ge2$, assign a second pair to the other of
rows $0,1$, and any remaining pairs to the first of these rows. With no
distinguished pair, use row $0$ for a sole pair; for $s\ge2$, assign the
first two pairs to rows $0,1$ and the rest to row $0$. If $s=0$, all rows
are constant. The distinguished column is excluded from row $0$ in each case.

Row $2$ contains the first endpoint of every pair. A fixed selected column
connects all selected supports. If there is no such column and $b\ge2$,
then $s=b\ge2$. Each exceptional row shares an endpoint with row $2$
through a pair whose exceptional row is the other of $0,1$. All other rows
agree with row $2$, proving intersection connectivity.
For complements, a fixed unselected column again connects all rows.
Otherwise every unpaired column was selected, so the complementary size
is $s$; if this size is at least two, the same argument applies to the
second endpoints on row $2$. Both row families therefore have the required
connectivity, and the column totals are the divergences $\pm1$ modulo $m$.
\end{proof}

\begin{corollary}[Parity criterion for unit column sums]\label{cor:selection-criterion}
Assume every column is covered, all supports have size at least two, and
all row quotas are half quotas. For $m\ge3$, a unit-total selection exists
if and only if either $m$ is odd or every incidence component has an even
number of columns.
\end{corollary}
\begin{proof}
Odd $m$ is Corollary~\ref{cor:odd-selector}. At even $m$, necessity follows
by summing the odd column totals in each component, and sufficiency is
Proposition~\ref{prop:pinned-selection} with no distinguished columns.
\end{proof}

\begin{theorem}[Sharp nonzero-deviation dichotomy]\label{thm:radius-dichotomy}
Assume that $\cO\ne\varnothing$, every column is covered, and every
support has size at least two. Among all three half-quota row profiles
for which every integer deviation $D_\lambda$ is nonzero, the minimum
of $\max_\lambda|D_\lambda|$ equals
\[
 \begin{cases}
 1,&\text{if every incidence component has an even number of columns},\\
 2,&\text{otherwise}.
 \end{cases}
\]
\end{theorem}
\begin{proof}
In each incidence component the sum of the deviations is zero, since
each block contributes $b_B+b_B-2b_B=0$. An odd number of values from
$\{-1,1\}$ cannot have sum zero, proving the obstruction to radius one.
If all components are even, use the directed pairs in the proof of
Proposition~\ref{prop:pinned-selection}. Put a pair's exceptional choice on row $0$
or $1$ and its ordinary choice on the third profile. Its exact integer
contribution is $-1$ at the tail and $+1$ at the head. Common fillers
have deviation zero, and the oriented divergence at every column is
$\pm1$. The three profiles meet their half quotas. Finally,
Theorem~\ref{thm:integer-selector} always gives radius at most two.
These bounds and the parity obstruction prove both cases.
\end{proof}

This integer dichotomy identifies the parity constraint underlying the
modular selection. The prescribed exclusions and intersection
connectivity in Proposition~\ref{prop:pinned-selection} supply the further
conditions needed for iteration.

\begin{remark}\label{rem:selector-limits}
The three hypotheses control different features of the selection. For $m=2$,
four columns and two quota-two rows with odd column totals force two disjoint
row supports: parity alone does not ensure intersection connectivity. A support
of size two with both columns distinguished leaves no entry available on the
excluded row. For general row quotas, beyond the half-quota setting of
Proposition~\ref{prop:pinned-selection}, eight columns in one quota-one block at
$m=4$ require at least eight selected entries, whereas the four rows
supply only four. These are obstructions to connectivity, prescribed exclusions, and row
capacity, respectively; see Question~\ref{ques:quotas}.
\end{remark}

\begin{lemma}[Inheritance of incidence parity]\label{lem:inherit}
Suppose a cycle-consistent block factorization has an even number of
colour--cycle labels in each incidence component for every direction of multiplicity at least two. Apply Proposition~\ref{prop:pinned-selection} to one such direction, with a column
for every colour--cycle label. Every old cycle lifts to one cycle. The new
fibre blocks are cycle-consistent, and both children of multiplicity at least two again satisfy
the even-component condition. Unsplit directions retain their column-component
partitions.
\end{lemma}
\begin{proof}
The column totals are units, so Lemma~\ref{lem:lift} lifts every old cycle $C$
to one cycle on $C\times\Zm$. Directions and cycle labels are therefore
constant on each new fibre block over an old source.

Fix a child direction of multiplicity $a'\ge2$. An old block $B$ has $m$
sources, one for each row $t$, and gives $m$ new fibre-block vertices.
Their supports are $J_{B,t}$ in the selected child and $A_B\setminus J_{B,t}$
in the complementary child. If two such supports intersect, their new block
vertices are joined through a common cycle column. Intersection connectivity
therefore places all $m$ vertices from $B$ in one child component.
Consequently, the block vertices of every child component $K$ are a union
of complete sets of $m$ copies. Write their number as $m b_K$.

Each cycle column has odd degree in the child graph. In the selected child,
this degree is its selected total; in the complementary child, it is its
eligible total minus its selected total. The eligible total is a multiple
of $m$ by cycle-consistency, whereas the selected total is $\pm1$ modulo $m$.
Writing $\cO_K$ for the columns in $K$, double counting gives
\[
 a'm b_K=\sum_{\lambda\in\cO_K}\deg(\lambda)
       \equiv |\cO_K|\pmod2.
\]
The left side is even, so $|\cO_K|$ is even. Both selected and complementary counts are
nonzero modulo $m$, so every old cycle meets both child directions.

For an unsplit direction the new blocks over the sources of an old block have
identical supports. Replacing each block vertex by these copies preserves
the partition of column vertices into components.
\end{proof}

\subsection{All even dimensions}
Start from $T_m(d)$ with each colour using one labelled copy of its
positive $m$-cycle. As in the odd construction, the initial vertex set is
one cycle-consistent block, and every split reduces
$\sum_i(a_i-1)$ by one.

\begin{theorem}[Even-dimensional specialization]\label{thm:even-dim}
For every even $m\ge4$ and every even $d\ge2$, $D_d(m)$ has a Hamilton
decomposition.
\end{theorem}
\begin{proof}
In the one-block factorization of $T_m(d)$, the incidence component has
the even number $d$ of cycle columns. Use
Proposition~\ref{prop:pinned-selection} with no distinguished columns for each
split. Unit totals preserve Hamiltonicity, and
Lemma~\ref{lem:inherit} preserves cycle-consistency and evenness of every
incidence component for each direction that still needs splitting.
After $d-1$ splits all multiplicities are one.
\end{proof}

The remaining case is even modulus in odd dimension. Its odd number
of Hamilton factors would give an odd incidence component at the
one-coordinate starting point. We resolve this parity constraint by
allowing one additional cycle and preserving a recolouring through
the subsequent lifts.

\section{Relative cyclic lifts and recolourings}\label{sec:relative}
On cycles meeting the recolouring support, first-return permutations
determine the effect of reconnection. A relative lift preserves these
permutations while lengthening the intervening paths. The interval
containing the nonzero voltages expands to contain all new nonzero fibre
points, so it remains available for the next split.

\subsection{First returns and source recolourings}
Let $U$ be a subset of the domain of a permutation $S$. On every cycle meeting
$U$, let $\alpha=\Ret_U(S)$ send a marked vertex to the next marked vertex. Write
$\ell(u)$ for the positive number of steps from $u$ to $\alpha(u)$, and let
$h_U(S)$ count the cycles of $S$ disjoint from $U$. When $U=\varnothing$, the
return permutation is the empty permutation and has zero cycles. The
\emph{open gap} from $u$ to $\alpha(u)$ is the set of intermediate vertices;
its two marked endpoints are excluded.

\begin{lemma}[Source reconnection]\label{lem:surgery}
If $r$ permutes $U$ and fixes its complement, then
\begin{equation}\label{eq:surgery}
 \Ret_U(Sr)=\alpha r,\qquad
 \cyc(Sr)=h_U(S)+\cyc(\alpha r).
\end{equation}
The new return time departing from $u$ is $\ell(r(u))$. In particular, a cycle
$D$ of $\alpha r$ corresponds to a cycle of $Sr$ of length
\begin{equation}\label{eq:weighted-surgery}
 L_D=\sum_{u\in D}\ell(r(u)).
\end{equation}
\end{lemma}
\begin{proof}
The first edge from $u$ goes to $S(r(u))$. Until the next marked vertex, all
following sources are outside $U$, so the old path following $r(u)$ is traversed
unchanged. These paths account for every vertex of each old cycle meeting $U$.
The cycles disjoint from $U$ remain untouched. Summing the lengths of the
concatenated paths gives~\eqref{eq:weighted-surgery}.
\end{proof}

\begin{proposition}[Complete boundary invariant for source reconnection]\label{prop:boundary-complete}
Let $S$ and $T$ be permutations of finite sets containing the same marked
set $U$. Extend each $r\in\Sym(U)$ by the identity on the appropriate
complement. The following are equivalent:
\begin{enumerate}
\item $\cyc(Sr)=\cyc(Tr)$ for every $r\in\Sym(U)$;
\item $\Ret_U(S)=\Ret_U(T)$ and $h_U(S)=h_U(T)$.
\end{enumerate}
The ambient sets and the intervening path lengths need not be equal.
\end{proposition}
\begin{proof}
Put $\alpha=\Ret_U(S)$ and $\beta=\Ret_U(T)$.
Lemma~\ref{lem:surgery} proves (ii)$\Rightarrow$(i).
Conversely, the maximum of $\cyc(Sr)$ over all $r$ is $h_U(S)+|U|$,
attained at $r=\alpha^{-1}$. Equality of the two functions therefore gives
$h_U(S)=h_U(T)$. Evaluating at $\alpha^{-1}$ gives
$\cyc(\beta\alpha^{-1})=|U|$, whence $\beta=\alpha$.
The same argument includes the empty marked set.
\end{proof}

The pair $(\alpha,\ell)$ is the \emph{marked return data}: $\alpha$
determines how the return paths are joined, and $\ell$ determines the
resulting lengths. Cycles disjoint from $U$ remain unchanged. For a
recolouring of a prescribed digraph, the available outgoing arcs also
restrict the source permutations; Proposition~\ref{prop:transport}
preserves these restrictions under lifting.

Fix a set $\cA$ of distinguished colours. A \emph{valid recolouring supported on $U$}
consists of permutations $S'_c$ obtained, at each $x\in U$, by permuting the
outgoing arcs of the colours in $\cA$. All remaining colours and all sources
outside $U$ are fixed. Thus there are $\sigma_x\in\Sym(\cA)$ with
\begin{equation}\label{eq:recolour}
 S'_c(x)=S_{\sigma_x(c)}(x),\qquad r_c=S_c^{-1}S'_c.
\end{equation}
Validity includes the bijectivity of every new colour. It follows that each
$r_c$ is a permutation supported on $U$.

\begin{proposition}[Two-colour sign obstruction]\label{prop:sign}
A valid exchange of two colours preserves the product of their permutation
signs. Consequently, for even $m$ and odd $d\ge3$, no sequence of such exchanges
transforms the canonical coordinate factorization of $D_d(m)$ into a Hamilton
decomposition.
\end{proposition}
\begin{proof}
For two factors $P,Q$, a valid exchange on $U$ has $P(U)=Q(U)$. Put
$q=P^{-1}Q$ on $U$ and extend it by the identity elsewhere. The new factors are
$Pq,Qq^{-1}$, whose sign product equals that of $P,Q$.
A coordinate translation on $\Zm^d$ has $m^{d-1}$ cycles of length $m$, so its
sign is $(-1)^{(m-1)m^{d-1}}=1$ for every even $m$ and every $d\ge2$.
A Hamilton permutation on the even number $m^d$ of vertices has sign $-1$.
For odd $d$, the canonical sign product is therefore $+1$ and the Hamilton
sign product is $-1$.
\end{proof}

\begin{proposition}[Cycle excess and sign]\label{prop:cycle-excess}
Let $S_1,\ldots,S_d$ be permutations of an $N$-element set, and define their
\emph{cycle excess} by
\[
 E=\sum_{c=1}^d\bigl(\cyc(S_c)-1\bigr).
\]
Then
\begin{equation}\label{eq:cycle-excess-sign}
 \prod_{c=1}^d\sgn(S_c)=(-1)^{d(N-1)}(-1)^E.
\end{equation}
For even $m$ and odd $d\ge3$, a factorization of $D_d(m)$ in the sign class of
the canonical coordinate factorization therefore has odd cycle excess. The
least nonnegative excess allowed by this parity condition is $1$.
\end{proposition}
\begin{proof}
A permutation with $k$ cycles on $N$ vertices has sign $(-1)^{N-k}$.
Multiplication over the colours gives~\eqref{eq:cycle-excess-sign}. In the torus
case the canonical sign product is $+1$, whereas $(-1)^{d(m^d-1)}=-1$.
Thus $E$ is odd and nonnegative. The value $1$ is the least value allowed by
this parity condition.
\end{proof}

\begin{corollary}[Preservation of cycle counts]\label{cor:inventory}
If every old colour-cycle has unit total voltage, a cyclic coordinate
lift preserves the number of cycles in each colour and the total cycle
excess $E$.
\end{corollary}
\begin{proof}
Lemma~\ref{lem:lift} replaces each old cycle by exactly one lifted cycle;
the number of colours is unchanged.
\end{proof}

\subsection{First returns to the zero section}
For an additive lift, partial voltages determine when the marked
vertices are visited on the zero section. Computing these times gives
both the marked order and the increments in return lengths.

\begin{proposition}[Zero-section order]\label{prop:phase-order}
Let $m\ge2$, let $S$ be a cycle of length $L$, and write $x_s=S^s x_0$ for $0\le s<L$. Let $U=\{x_{s_0},\ldots,x_{s_{r-1}}\}$, where
$0=s_0<s_1<\cdots<s_{r-1}<L$. For a lift
$\widehat S(x,t)=(Sx,t+\delta(x))$ over $\Zm$, suppose
$K=\sum_{s=0}^{L-1}\delta(x_s)$ is a unit. Put
\[
 P_j=\sum_{s=0}^{s_j-1}\delta(x_s),\qquad
 \kappa_j=[-K^{-1}P_j]_m\in\{0,\ldots,m-1\},
\]
where $[\cdot]_m$ denotes the least nonnegative representative.
The unique visit to $(x_{s_j},0)$ in the first $mL$ steps from $(x_0,0)$
occurs at
\begin{equation}\label{eq:zero-visit-time}
 T_j=s_j+L\kappa_j.
\end{equation}
The first-return order on $U\times\{0\}$ equals that on $U$ if and only if
\begin{equation}\label{eq:phase-order}
 0=\kappa_0\le\kappa_1\le\cdots\le\kappa_{r-1}.
\end{equation}
In that case, writing $\ell_j$ for the old return times, the new times are
$\ell_j+q_jL$, where
\begin{equation}\label{eq:phase-budget}
 q_j=\kappa_{j+1}-\kappa_j\quad(0\le j<r-1),\qquad
 q_{r-1}=m-1-\kappa_{r-1}.
\end{equation}
In particular $q_j\ge0$ and $\sum_jq_j=m-1$.
\end{proposition}
\begin{proof}
At time $s_j+kL$ the fibre label is $P_j+kK$. Exactly one
$k\in\{0,\ldots,m-1\}$ makes it zero, namely $\kappa_j$.
The quantities $s_j$ lie in $[0,L)$, so sorting the times $T_j$ is the
same as sorting the pairs $(\kappa_j,s_j)$ lexicographically. This proves
the order criterion. Subtract successive times, using $mL$ as the final
return time to $(x_0,0)$, to obtain~\eqref{eq:phase-budget}.
\end{proof}

Appendix~\ref{app:capacity} uses the same calculation to prescribe
return times. For the repeated splitting argument, we concentrate the
additional turns in one interval.

\subsection{Lifting over a single return interval}

\begin{theorem}[Lifting over a single return interval]\label{thm:onegap}
Let $S$ be a cycle of length $L$, let $\varnothing\ne U\subseteq V(S)$, and let
$\alpha=\Ret_U(S)$. Fix one gap from $u_*$ to $\alpha(u_*)$. Let $F$ be a set of
size $q$ with distinguished element $0$. Choose permutations $g_x\in\Sym(F)$
that are the identity outside the open distinguished gap; in particular,
they fix every fibre label at both marked endpoints. Suppose the ordered product of the $g_x$ along that gap is a $q$-cycle.
Then
\[
 \widehat S(x,t)=(Sx,g_x(t))
\]
is a cycle of length $qL$. On $U\times\{0\}$, naturally identified with $U$,
its first return is $\alpha$, and its return times are
\begin{equation}\label{eq:onegap-roof}
 \widehat\ell(u)=\ell(u)+(q-1)L\one_{\{u=u_*\}}.
\end{equation}
Every point with nonzero fibre label lies in the open distinguished lifted gap.
\end{theorem}
\begin{proof}
Write $H$ for the ordered product of the fibre maps along the distinguished gap.
An ordinary gap has identity fibre transport and connects its zero-labelled
endpoints directly. Starting from $(u_*,0)$, the first passage through the
distinguished gap reaches its endpoint with label $H(0)$. Each further complete
old cycle applies $H$ once more. The labels
\[
 H(0),H^2(0),\ldots,H^{q-1}(0),H^q(0)=0
\]
are nonzero until the last one. Outside the open distinguished gap, all marked
vertices therefore have nonzero labels during these intermediate traversals.
The first return to a zero-labelled marked vertex therefore adds
$q-1$ complete traversals of the old cycle, proving~\eqref{eq:onegap-roof}.

The monodromy $H$ is cyclic, so the lift has one orbit. Equivalently, the stated
return times sum to $qL$ and exhaust the lifted vertex set. All ordinary lifted
gaps consist only of zero-labelled vertices. Every point with a different label
therefore belongs to the remaining open gap. The argument uses the ordered
product of the maps $g_x$ and applies to noncommuting fibre permutations.
\end{proof}

\begin{proposition}[The return-time identity]\label{prop:return-budget}
Let $S$ be a cycle of length $L$, let $\varnothing\ne U\subseteq V(S)$, and put
$\alpha=\Ret_U(S)$. Suppose a lift
$\widehat S(x,t)=(Sx,g_x(t))$ over a $q$-element fibre is a single $qL$-cycle and
has first return $\alpha$ on $U\times\{0\}$. Then there are nonnegative integers
$k_u$, $u\in U$, such that
\begin{equation}\label{eq:return-budget}
 \widehat\ell(u)=\ell(u)+k_uL,
 \qquad \sum_{u\in U}k_u=q-1.
\end{equation}
\end{proposition}
\begin{proof}
Projection to the base cycle shows that any return from $(u,0)$ to
$(\alpha(u),0)$ has length $\ell(u)$ modulo $L$ and is at least $\ell(u)$.
This gives the first identity with $k_u\ge0$. The first-return paths partition
the lifted cycle, so their lengths sum to $qL$, whereas
$\sum_{u\in U}\ell(u)=L$. Subtracting proves the second identity.
\end{proof}

The single-interval lift assigns all $q-1$ additional traversals of the
base cycle to the return departing from $u_*$ (Figure~\ref{fig:onegap}).
Every nonzero fibre point lies in this enlarged interval, which provides
the positions for the next split.

\begin{figure}[htbp]
\centering
\[
\begin{array}{c@{\quad}c@{\quad}c}
 u_* & \xrightarrow{\quad\ell(u_*)\quad} & \alpha(u_*)\\[-1pt]
 \big\downarrow & & \big\downarrow\\[-1pt]
 (u_*,0) & \xrightarrow{\quad\ell(u_*)+(q-1)L\quad} & (\alpha(u_*),0)
\end{array}
\]
\vspace{-5pt}
\[
 \widehat\ell(u)=\ell(u)\qquad(u\in U\setminus\{u_*\}).
\]
\caption{First-return paths under the single-interval lift. Horizontal
arrows are labelled by their lengths; vertical arrows place the marked
vertices on the zero section. The marked cyclic order is preserved.}
\label{fig:onegap}
\end{figure}

\begin{corollary}[Additive return-preserving lifting]\label{cor:additive}
In Theorem~\ref{thm:onegap}, take $F=\Zm$ and $g_x(t)=t+\delta(x)$. It suffices
that all nonzero voltages lie in the open distinguished gap and that
$\sum_x\delta(x)$ is a unit modulo $m$. In the binary case, any number of selected sources is allowed, provided
they lie in this interval and their total is a unit modulo $m$.
\end{corollary}
\begin{proof}
Translation by a unit is an $m$-cycle on $\Zm$.
\end{proof}

\begin{example}[A cyclic lift with changed marked order]\label{ex:spread}\label{ex:unit-not-order}
Let $S=(0\,1\,\cdots\,11)$,
$U=\{0,4,8\}$, and $m=4$. Give sources $1,5,9,10,11$ voltage one and all
others voltage zero. The total is $K=1$ modulo four, but the marked
phases are $(0,3,2)$. Formula~\eqref{eq:zero-visit-time} gives visit times
$(0,40,32)$, so the zero-section order is $(0,8,4)$ instead of $(0,4,8)$.
\end{example}

\begin{proposition}[Lifting a recolouring]\label{prop:transport}
Let $S_c$ be a permutation factorization with distinguished colour set $\cA$ and marked
set $U$. On each old distinguished cycle meeting $U$, assume that every nonzero
voltage lies in the open interior of one $U$-gap and that the total voltage is a
unit. Give every other old cycle unit total as well. Then every valid
distinguished-colour recolouring supported on $U$ transports to its single copy
$U\times\{0\}$ in the lift, with the same number of cycles in each colour.
When the lift satisfies~\eqref{eq:physical-delta}, the transported recolouring
uses only the existing labelled positive coordinate arcs.
\end{proposition}
\begin{proof}
Write the recolouring as in~\eqref{eq:recolour}. Every distinguished voltage vanishes on
$U$. At $(x,0)$ the new colour-$c$ head is
\[
 (S_{\sigma_x(c)}x,0)=(S_cr_cx,0),
\]
which is the old colour-$c$ head from $(r_cx,0)$. The lifted relative source
permutation is $r_c$ on the selected copy of $U$. Hence the new lifted factors
are bijective, and the recolouring permutes the existing arcs at each source.

Theorem~\ref{thm:onegap} preserves the first return on every old distinguished cycle
meeting $U$. The recoloured first return is consequently $\alpha_cr_c$ both before
and after the lift. Every old cycle disjoint from $U$ lifts to one cycle, which is unaffected by
the transported recolouring. Lemma~\ref{lem:surgery}, including its unmarked-cycle term, proves the
claim. Colours outside the distinguished set are unchanged.
\end{proof}

\begin{proposition}[Lengths after transported recolouring]\label{prop:transported-lengths}
Under the hypotheses of Proposition~\ref{prop:transport}, fix a distinguished
colour and suppress its subscript. Let $\alpha=\Ret_U(S)$, and let $r$ be its
relative source permutation. For every old cycle $C$ meeting $U$, write
$u_C$ for the initial marked point of the distinguished return interval.
If $D$ is a cycle of $\alpha r$, the corresponding recoloured cycle has lengths
\begin{equation}\label{eq:transported-lengths}
 \begin{aligned}
 L_D&=\sum_{u\in D}\ell(r(u)),\\
 \widehat L_D&=L_D+(m-1)
       \sum_{\substack{C\cap U\ne\varnothing\\u_C\in r(D)}}|C|.
 \end{aligned}
\end{equation}
\end{proposition}
\begin{proof}
Apply~\eqref{eq:weighted-surgery} before and after the lift. By
\eqref{eq:onegap-roof}, the only changed return time on $C$ is that departing
from $u_C$, which increases by $(m-1)|C|$. Summing over $r(D)$ gives the formula.
\end{proof}

Formula~\eqref{eq:transported-lengths} assigns each added traversal of
an old cycle to the new cycle containing its distinguished return path.
For example, take a $6$-cycle with marks $0,2,4$, unit voltage at source
$1$, and source transposition $(0\,2)$. Over $\mathbb Z_4$, the two
recoloured cycle lengths change from $2,4$ to $20,4$: all $18$ additional
steps enter the first cycle. For a Hamilton recolouring, every return
path belongs to the single spanning cycle.

\subsection{Admissible marked factorizations}
The induction preserves both the incidence conditions for selection and
the interval condition for first returns. In the following definition,
the fibre coordinate is that of the most recent split in
\eqref{eq:sum-sheet}.

\begin{definition}\label{def:collar}
An \emph{admissible marked factorization} consists of a factorization of
$T_m(a_1,\ldots,a_n)$, a set of distinguished colours $\cA$, a marked source set $U$, and an
identification $X=Y\times\Zm$ satisfying the following conditions.
\begin{enumerate}
\item Each fibre $\{y\}\times\Zm$ is cycle-consistent.
\item At each source, each coordinate direction has at most one distinguished colour.
\item The set $U$ is contained in $Y\times\{0\}$. In every distinguished cycle meeting
$U$, all vertices with nonzero last fibre coordinate lie in one open $U$-gap.
\item For every direction with multiplicity at least two, every component of the
block--cycle incidence graph on the full set~\eqref{eq:orbit-set} has an even
number of cycle-column vertices.
\end{enumerate}
A valid recolouring of the distinguished colours supported on $U$ may also be
specified. The construction of the initial marked factorization determines its
cycle counts; the theorem below preserves them under every subsequent split.
\end{definition}

\begin{theorem}[Relative splitting theorem]\label{thm:closure}
Let $m\ge4$ be even. In an admissible marked factorization, any direction of multiplicity
$a\ge2$ can be split into $\lceil a/2\rceil$ and $\lfloor a/2\rfloor$ so that the
result is again an admissible marked factorization. The marked set becomes $U\times\{0\}$,
and every valid distinguished recolouring on $U$ retains its colourwise cycle count.
If the multiplicity sum is $d$ and there are $n$ coordinates, exactly
$d-n$ further splits yield $(1,\ldots,1)$. An initial Hamilton recolouring
remains Hamilton after these splits.
\end{theorem}
\begin{proof}
For the direction being split, take one column for every colour--cycle label
and one block for every current fibre. Mark the columns of the distinguished colours.
By Definition~\ref{def:collar}(ii), each support contains at most one distinguished
column. Proposition~\ref{prop:pinned-selection} gives exact row quotas and unit cycle
totals, while excluding every distinguished choice from the old zero row. Lemma~\ref{lem:physical}
realizes this selection by positive coordinate arcs.

Fix a distinguished cycle $C$ meeting $U$, and let $G_C$ be the open gap
in Definition~\ref{def:collar}(iii). Write a source $v\in C$ as $v=(y,t)$ in
the current fibre coordinates. The exclusion in
Proposition~\ref{prop:pinned-selection} and admissibility give, in order,
\[
 \delta_c(v)=1\quad\Longrightarrow\quad t\ne0
             \quad\Longrightarrow\quad v\in G_C.
\]
Thus all nonzero voltages on $C$ lie in the same open gap and vanish at its
marked endpoints. Their total is a unit, so Theorem~\ref{thm:onegap}
preserves the first return on $U\times\{0\}$. It also places every vertex
with nonzero \emph{new} fibre coordinate in one enlarged open gap. This is
exactly condition~(iii) for the next application of the selection theorem.

Proposition~\ref{prop:transport} now transports every valid distinguished
recolouring. Lemma~\ref{lem:inherit} preserves fibre consistency and
componentwise evenness, while each child inherits at most one distinguished
colour from its parent. The marked set lies on the new zero section.
All four conditions of Definition~\ref{def:collar} therefore hold after the split.

Finally, the nonnegative integer
\[
 \sum_{i=1}^n(a_i-1)=d-n
\]
decreases by one at each split. It vanishes exactly when every multiplicity is
one. The marked first-return permutations and the relative source permutations
are therefore unchanged throughout the refinement. The return times evolve
according to~\eqref{eq:onegap-roof}.
\end{proof}

We apply Theorem~\ref{thm:closure} with three distinguished colours
whose factors contain four cycles in total. Together with an even
number of auxiliary Hamilton factors, they give an even number of
incidence columns. Their recolouring remains supported on a single
copy of the original marked set throughout the refinement.

\section{A three-dimensional initial factorization}\label{sec:seed}
The initial factorization combines three distinguished colours with an
even number of auxiliary Hamilton factors. Throughout
Sections~\ref{sec:seed}--\ref{sec:entry},
$m\ge4$ is even, $p\ge2$, and $d=2p+3$. Three colours $F_0,F_1,F_2$ will be distinguished;
the other $2p$ colours are denoted by $A_i,B_i$ for $0\le i<p$. The geometric
coordinates are $(x,y,w)\in\Zm^3$, and $h=x+y+w$ is represented in
$\{0,\ldots,m-1\}$ when an inequality involving $h$ is used.

\subsection{Three distinguished colours}
Define the directions of $F_0,F_1,F_2$ by Table~\ref{tab:nearcore}.
\begin{table}[htbp]
\centering
\caption{The three-colour initial factorization.}\label{tab:nearcore}
\begin{tabular}{lccc}
\toprule
Condition & $F_0$ & $F_1$ & $F_2$\\
\midrule
$h=0$ & $e_w$ & $e_y$ & $e_x$\\
$h=1$, $w\in\{0,1\}$ & $e_x$ & $e_w$ & $e_y$\\
$h=1$, $w\notin\{0,1\}$ & $e_y$ & $e_w$ & $e_x$\\
$h=2$, $w=0$ & $e_y$ & $e_x$ & $e_w$\\
All other vertices & $e_x$ & $e_y$ & $e_w$\\
\bottomrule
\end{tabular}
\end{table}

\begin{lemma}\label{lem:nearcore}
Table~\ref{tab:nearcore} is a factorization of $D_3(m)$. The factors $F_0,F_1$
are Hamilton, while $F_2$ has exactly two cycles, each of length $m^3/2$.
Their vertex sets are distinguished by
\begin{equation}\label{eq:defect-label}
 \ep(x,y,w)=w-\max(h-2,0)\pmod2.
\end{equation}
For each fixed $(h,w)$, the $m$ vertices obtained by varying $x$ form a
cycle-consistent block.
\end{lemma}
\begin{proof}
Start on two coordinates $(a,w)$ with height $h=a+w$ and multiplicities $(2,1)$.
Let $F_0$ use $e_w$ only at height $0$, let $F_1$ use it only at height $1$, and
let $F_2$ use it at all other heights. All remaining steps use $e_a$. Each layer
map is a translation. The $m$-step returns on a height layer translate $w$ by
$1,1,m-2$, respectively. Thus the first two colours are Hamilton and the third
has two cycles. The function~\eqref{eq:defect-label} is constant along the third
colour and distinguishes its two cycles: away from the height wrap this follows
by direct substitution, and across the height wrap the change is even, being $m-2$.

Split $a=x+y$, using $y$ as the new fibre coordinate. Among the two colours using
$a$, select $F_1$ at $h=0$; select $F_2$ at $h=1,w=0,1$; select $F_0$ at the other
points of $h=1$ and at $h=2,w=0$; and select $F_1$ elsewhere. The total on $F_0$
is $m-1$, and that on $F_1$ is $m^2-m-1$. Both are $-1$ modulo $m$. At $h=1$ the
two $F_2$ cycles are distinguished by the parity of $w$, so each receives exactly
one selected source. Lemma~\ref{lem:lift} proves the asserted lifted cycle lengths.
Lemma~\ref{lem:physical} gives exactly Table~\ref{tab:nearcore}.

The new fibres have fixed $(a,w)$, equivalently fixed $(h,w)$. Each lies in one
lifted cycle of every colour and has constant directions.
\end{proof}

\subsection{The auxiliary Hamilton factors}
Put $k=\lfloor p/2\rfloor$. On two coordinates $(a,w)$, specify the colours
using $e_w$ at height $h$ as follows:
\begin{equation}\label{eq:shell-W}
 W_0=\{B_0,A_1,\ldots,A_{p-1}\},\quad
 W_1=\{A_0,B_1,\ldots,B_{p-1}\},\quad
 W_h=\{A_0,\ldots,A_{p-1}\}\ (h\ge2).
\end{equation}
The complement $V_h$ in the set of auxiliary colours has size $p$ and uses $e_a$.
Every $A_i$ uses $e_w$ on $m-1$ layers and every $B_i$ on one layer. Hence all
$2p$ factors are Hamilton on this two-coordinate multitorus.

To split $a=x+y$, choose $k$ colours from $V_h$ at each pair $(h,w)$. Use the
ordered disjoint pair lists in Table~\ref{tab:shellpairs}; the list is empty at
other heights. Empty index ranges are omitted. For each ordered pair $(u,v)$, the first and second endpoints are
$u$ and $v$, respectively. In pair number $j$,
starting with number $0$, choose its second endpoint at $w=1$ if $j=1$, and at
$w=0$ otherwise; choose its first endpoint on all other rows. Complete the selection to $k$ colours with the first unpaired columns in
the order
$A_0,\ldots,A_{p-1},B_0,\ldots,B_{p-1}$. The selected colours use $e_y$ and the
remaining members of $V_h$ use $e_x$.

\begin{table}[htbp]
\centering
\caption{Ordered pair lists for the set of auxiliary colours.}\label{tab:shellpairs}
\begin{tabular}{cl}
\toprule
Height and parity & Pair list\\
\midrule
$h=0$ & $(A_0,B_1)$\\
$h=1$, $p$ even & $(B_0,A_1),(A_2,A_3),\ldots,(A_{p-2},A_{p-1})$\\
$h=2$, $p$ even & $(B_2,B_3),(B_4,B_5),\ldots,(B_{p-2},B_{p-1})$\\
$h=1$, $p$ odd & $(A_1,A_2),(A_3,A_4),\ldots,(A_{p-2},A_{p-1})$\\
$h=2$, $p$ odd & $(B_0,B_2),(B_3,B_4),\ldots,(B_{p-2},B_{p-1})$\\
\bottomrule
\end{tabular}
\end{table}

\begin{lemma}\label{lem:shell}
This defines a Hamilton factorization of $T_m(p-k,k,p)$. Its direction functions
depend only on $(h,w)$.
\end{lemma}
\begin{proof}
The pair lists are disjoint within each layer, and enough unpaired columns
remain to complete each row selection. Every auxiliary colour appears in exactly
one pair in Table~\ref{tab:shellpairs}. That pair contributes either $1$ or $m-1$
to its second-split voltage. Columns selected on every row contribute
multiples of $m$. Thus every
old Hamilton factor has unit total voltage and lifts to one $m^3$-cycle. The row counts
are $p-k,k,p$, as required.
\end{proof}

Superpose Table~\ref{tab:nearcore} and the auxiliary factors, using distinct
parallel-arc labels. We obtain a factorization of
\begin{equation}\label{eq:seed-widths}
 T_m(p-k+1,k+1,p+1)
\end{equation}
with $d-1$ Hamilton factors and one factor consisting of two cycles. Its full cycle-column
set is
\begin{equation}\label{eq:seed-columns}
 \cO=\{F_0,F_1,F_{2,0},F_{2,1}\}
       \cup\{A_i,B_i:0\le i<p\},\qquad |\cO|=2p+4.
\end{equation}
Here $F_{2,j}$ labels the cycle of colour $F_2$ with value $j$
in~\eqref{eq:defect-label}. At every source, exactly one distinguished
colour uses each coordinate direction.

\subsection{Incidence geometry of the initial factorization}
\begin{lemma}\label{lem:seed-incidence}
For $p\ge4$, the full block--cycle incidence graph of each of the three
directions in~\eqref{eq:seed-widths} is connected. In particular it has an even
number of column vertices, and every old cycle meets every direction.
\end{lemma}
\begin{proof}
We first compute the auxiliary components. At a fixed height, the selected and
complementary row families in Table~\ref{tab:shellpairs} are intersection-connected, because
their sizes $k,p-k$ are at least two. A column common to every row connects
the family; otherwise the first two exceptional rows each intersect an
ordinary row. Consequently,
the union of all columns occurring in a direction at one height is connected.

To compute these unions, write
$A_{[r,s]}=\{A_j:r\le j\le s\}$ and similarly for $B$, with empty intervals
understood. At $h=0$, the $y$-union is
$\{A_0\}\cup B_{[1,k]}$, and the $x$-union is
$\{A_0,B_1\}\cup B_{[k+1,p-1]}$. At $h\ge3$, the two unions are
$B_{[0,k-1]}$ for $y$ and $B_{[k,p-1]}$ for $x$.
For the exceptional middle heights they are:
\begin{center}
\begin{tabular}{cll}
\toprule
& $y$-union & $x$-union\\
\midrule
$h=1$, $p$ even & $\{B_0\}\cup A_{[1,p-1]}$ & $\{B_0\}\cup A_{[1,p-1]}$\\
$h=2$, $p$ even & $\{B_0\}\cup B_{[2,p-1]}$ & $B_{[1,p-1]}$\\
$h=1$, $p$ odd & $A_{[1,p-1]}$ & $\{B_0\}\cup A_{[1,p-1]}$\\
$h=2$, $p$ odd & $\{B_0\}\cup B_{[2,p-1]}$ & $B_{[0,p-1]}$\\
\bottomrule
\end{tabular}
\end{center}
Thus the auxiliary $x$ graph is connected when $p$ is odd. When $p$ is even its
components are
\[
 \{B_0,A_1,\ldots,A_{p-1}\},\qquad
 \{A_0,B_1,\ldots,B_{p-1}\}.
\]
The distinguished colour $F_0$ joins them, using $x$ at $(h,w)=(1,0)$ and also at
height $3$. The auxiliary $y$ graph is connected for even $p$. For odd $p$ its
components are $A_{[1,p-1]}$ and $\{A_0\}\cup B_{[0,p-1]}$; $F_0$ joins them
through its $y$-steps at $(h,w)=(1,2)$ and $(2,0)$. These rows exist already at
$m=4$. The auxiliary $w$ graph is connected because the three sets
in~\eqref{eq:shell-W} overlap through $A_0$ and $A_1$.

Each distinguished cycle uses each direction. This is immediate for the two Hamilton
distinguished colours from the table. For $F_2$, both parity labels occur among the
$x$-rows at $h=0$, the $y$-rows at $h=1,w=0,1$, and the $w$-rows at $h=2$.
They therefore attach to the connected graphs just obtained. The full column
count is $2p+4$, which is even.
\end{proof}

\begin{lemma}[Initial factorizations in dimensions seven and nine]\label{lem:small-seed}
For $p=2,3$, the same three-colour factorization and auxiliary pair table
give~\eqref{eq:seed-widths}, with multiplicities $(2,2,3)$ and $(3,2,4)$,
respectively. Every component of each direction's full cycle-incidence graph
contains an even number of cycle-column vertices.
\end{lemma}
\begin{proof}
The proof of auxiliary Hamiltonicity uses only the unique pair of each colour,
and therefore also applies here. The $x$ and $w$ cycle-incidence graphs are connected.
For $p=2$, the $y$ components are
\[
 \{F_0,F_1,F_{2,1},A_0,B_0,B_1\},\qquad \{F_{2,0},A_1\};
\]
for $p=3$ they are
\[
 \{F_0,F_{2,1},A_1,B_2\},\quad
 \{F_1,A_0,B_0,B_1\},\quad \{F_{2,0},A_2\}.
\]
Substitution in the pair table gives the following auxiliary selections.
The unique auxiliary colour using $y$ for $p=2$ is $A_0$ at $h=0$
(replaced by $B_1$ at $w=0$), $B_0$ at $h=1$ (replaced by $A_1$ at $w=0$),
and $B_0$ at $h\ge2$. For $p=3$ it is $A_0,A_1,B_0$ at $h=0,1,2$, replaced at
$w=0$ by $B_1,A_2,B_2$, respectively, and is $B_0$ at $h\ge3$.
All other auxiliary colours outside $W_h$ use $x$. The displayed $y$ components
now follow by the two possible labels of $F_2$. For $x$, at $p=2$ the two
auxiliary pairs $\{B_0,A_1\}$ and $\{A_0,B_1\}$ are joined by $F_0$ at heights
$1$ and $3$. At $p=3$ the height-$2$ and generic supports connect all $B$ colours,
height $0$ attaches $A_0$, and height $1$ attaches $A_1,A_2$ through $B_0$.
The three sets $W_h$ connect the $w$ graph. Every distinguished cycle uses each direction,
so all its labels attach. Larger $m$ repeat the same supports and the two defect
parities; all indicated height and row types already exist at $m=4$.
\end{proof}

\section{An anchored Hamilton recolouring in dimension three}\label{sec:anchored}
We construct three Hamilton factors that agree with
Table~\ref{tab:nearcore} at four prescribed vertices. These agreements leave one
available source on each of the four initial distinguished cycles for the first
lift. The replacement permutes only the three distinguished labelled arcs and
therefore leaves the auxiliary Hamilton factors unchanged.

\subsection{A permutation-valued layer}
Directions and colours in this subsection are numbered $0,1,2$. Put
\[
 g_0=(0,0),\qquad g_1=(1,0),\qquad g_2=(0,1),
 \qquad R(a,b)=(-a-b,a).
\]
On the three-element colour set define
\[
 r(c)=c+1\pmod3,\qquad \tau_z(c)=2z-c\pmod3.
\]
Thus $R^3=\id$ and $r^3=\id$. The punctured lines
\[
 \{R^z(t,-t):t\in\Zm\setminus\{0\}\},\qquad z=0,1,2,
\]
are pairwise disjoint. They are the punctured diagonal, vertical, and horizontal
lines, respectively. Define a permutation-valued row on $\Zm^2$ by
\begin{equation}\label{eq:star-row}
 \sigma^*\bigl(R^z(t,-t)\bigr)=
 \begin{cases}
 r^{-1},&t=1,\\
 r,&t=-1,\\
 \tau_z,&t\notin\{0,1,-1\},
 \end{cases}
 \qquad
 \sigma^*(q)=\id\quad\text{at all other }q.
\end{equation}

\begin{lemma}[Cyclic-star layer]\label{lem:star}
Put the row~\eqref{eq:star-row} on one height layer of $D_3(m)$, identifying that
layer with the last two coordinates. On the other layers use constant rows
$r^s$, with the following numbers of occurrences:
\begin{equation}\label{eq:star-counts}
\begin{array}{c|ccc}
 & \id&r&r^{-1}\\\hline
3\nmid m&0&1&m-2\\
3\mid m&m-5&1&3.
\end{array}
\end{equation}
Then the three colours are Hamilton. The constant rows may be placed in any
order. The special row may be translated inside its layer, and the distinguished
height may be chosen arbitrarily.
\end{lemma}
\begin{proof}
Write $\tr_v$ for translation by $v$ on $\Zm^2$. The special layer map of colour
$c$ is $q\mapsto q+g_{\sigma^*(q)(c)}$. Subtracting the usual translation $g_c$
gives the permutation
\[
 P_c(q)=q+g_{\sigma^*(q)(c)}-g_c.
\]
For $c=0$, substitution in~\eqref{eq:star-row} shows that $P_0$ fixes every point
outside the following cycle of length $2m$:
\begin{equation}\label{eq:star-support}
 \bigl(C_+,X_1,X_2,\ldots,X_{m-1},C_-,Y_1,Y_2,\ldots,Y_{m-1}\bigr),
\end{equation}
where
\[
 C_+=(1,-1),\quad C_-=(-1,1),\quad X_s=(s,0),\quad Y_s=(0,s).
\]
The rows satisfy
$\sigma^*(Rq)(c+1)=\sigma^*(q)(c)+1$, and
$Rg_j=g_{j+1}-g_1$, with indices modulo three. Hence
\begin{equation}\label{eq:star-equivariance}
 P_{c+1}R=RP_c.
\end{equation}
In particular, all three special layer maps are bijective. The other layer maps
are translations, so the entire direction assignment is a factorization.

Take the height return starting immediately before the special layer. It is
$\tr_{v_c}P_c$, where
\begin{equation}\label{eq:star-voltage}
\begin{array}{c|ccc}
 &v_0&v_1&v_2\\\hline
3\nmid m&(1,-2)&(1,1)&(-2,1)\\
3\mid m&(1,3)&(-4,1)&(3,-4).
\end{array}
\end{equation}
These vectors include the translation $g_c$ from the special layer itself.
They satisfy $v_{c+1}=Rv_c$, so~\eqref{eq:star-equivariance} reduces the proof to
colour zero. Appendix~\ref{app:star} proves directly that $\tr_{v_0}P_0$ is a
single $m^2$-cycle in both cases, including $m=4$ and $m=6$. The height return
then gives a cycle of length $m^3$ for each colour.

The ordinary layer translations commute, which proves order-independence of the
constant rows. Translating the special layer conjugates $P_c$ by a translation
commuting with $\tr_{v_c}$. Changing the distinguished height changes only the
starting section of the cyclic sequence of layer maps.
\end{proof}

\subsection{Agreement at four prescribed vertices}
Set
\begin{equation}\label{eq:anchors}
 Q_m=\{(0,1,0),(1,1,0),(0,0,0),(0,m-2,3)\},
\end{equation}
and let $\pi=(1\ 2)$ permute the three colours. We describe the replacement
factorization $H_0,H_1,H_2$ in the physical coordinates $(x,y,w)$.
At height $h=1$ put
\[
 \sigma_{(x,y,w)}=\sigma^*(y,w-3).
\]
At every other height put $\sigma_{(x,y,w)}=r^{s_h}$, where, if $3\nmid m$,
\begin{equation}\label{eq:constant-rows-one}
 s_2=1,\qquad s_h=2\quad(h\notin\{1,2\}),
\end{equation}
and, if $3\mid m$,
\begin{equation}\label{eq:constant-rows-two}
 s_2=1,\qquad s_0=s_3=s_4=2,\qquad s_h=0\quad(5\le h<m).
\end{equation}
The height $1$ is always treated by the special row. Define
\begin{equation}\label{eq:terminal-core}
 H_c(x,y,w)=(x,y,w)+e_{\sigma_{(x,y,w)}(\pi(c))},
 \qquad(e_0,e_1,e_2)=(e_x,e_y,e_w).
\end{equation}

\begin{proposition}[Anchored three-colour decomposition]\label{prop:anchor}
For every even $m\ge4$, the factors~\eqref{eq:terminal-core} form a Hamilton
decomposition of $D_3(m)$. Their direction functions agree with those of
$F_0,F_1,F_2$ at all four points of $Q_m$.
\end{proposition}
\begin{proof}
Equations~\eqref{eq:constant-rows-one} and~\eqref{eq:constant-rows-two} have
exactly the layer counts in~\eqref{eq:star-counts}. Lemma~\ref{lem:star}, followed
by the colour permutation $\pi$, proves Hamiltonicity.

Table~\ref{tab:anchors} records the four agreements. At the first point,
$(y,w-3)=(1,-3)$ lies off the three punctured lines: neither coordinate is zero
and their sum is $-2\ne0$ in $\Zm$. At the last point,
$(y,w-3)=(-2,0)=R^2(2,-2)$, and $2\notin\{0,1,-1\}$ for every even $m\ge4$.
It therefore uses the generic row $\tau_2$. The other two points lie on
constant layers. The resulting rows are exactly those of
Table~\ref{tab:nearcore}.
\end{proof}

\begin{table}[htbp]
\centering
\caption{Agreement between the initial and Hamilton factorizations. Entries are
coordinate indices $0,1,2$.}\label{tab:anchors}
\begin{tabular}{cccc}
\toprule
Source & Height & Row before $\pi$ & Final direction row\\
\midrule
$(0,1,0)$ & $1$ & $\id$ & $(0,2,1)$\\
$(1,1,0)$ & $2$ & $r$ & $(1,0,2)$\\
$(0,0,0)$ & $0$ & $r^{-1}$ & $(2,1,0)$\\
$(0,m-2,3)$ & $1$ & $\tau_2$ & $(1,2,0)$\\
\bottomrule
\end{tabular}
\end{table}

Let $U_m$ be the sources where the full three-colour direction row changes. Since both
old and new three-colour factorizations use exactly one copy of each coordinate arc,
the replacement is a valid simultaneous three-colour recolouring supported on
$U_m$, and $U_m\cap Q_m=\varnothing$. In the superposition with the auxiliary
factors, the auxiliary labelled arcs remain unchanged.

\begin{proposition}\label{prop:patch-size}
The number of sources in $U_m$ is
\begin{equation}\label{eq:patch-size}
 |U_m|=m^3-m^2-4m+7.
\end{equation}
\end{proposition}
\begin{proof}
At height $0$ the two three-colour assignments agree everywhere, giving $m^2$ sources.
At height $2$ they agree exactly when $w=0$, giving $m$ more. At height $1$, for
each of $w=0,1$, the two intersections with the punctured star are the only
disagreements, giving $2(m-2)$ agreements. On $w=3$ the generic horizontal-star
rows give $m-3$ agreements. At every remaining height-$1$ source the rows differ: the replacement row
is $\pi$ outside the star, a different reflected row at a turn, and a
different row on each generic nonhorizontal branch. At the remaining heights the
constant rows are different from the initial three-colour row. Thus the total agreement
count is $m^2+m+2(m-2)+(m-3)=m^2+4m-7$.
\end{proof}

\section{Initial marked factorizations}\label{sec:entry}
A first coordinate split turns the three-dimensional factorization into
an admissible marked factorization. Its distinguished voltages are
supported on $Q_m$, where the Hamilton recolouring agrees with the
original factors. Choosing one voltage on each distinguished cycle
creates the return intervals required by Theorem~\ref{thm:closure}.

\subsection{A matched first split}
Consider blocks of $m$ repeated rows, each with exactly one distinguished cycle label
and $a-1$ auxiliary colours, and let $b=\lfloor a/2\rfloor$. For every distinguished
cycle choose a different eligible block, one row of that block, and an auxiliary
\emph{partner} present there. Assume that the partners are distinct. In a chosen block, one of the selected entries is the distinguished colour
on its designated row and the partner on the other $m-1$ rows. Delete the partner from the residual support of that
block and reduce its residual quota by one. Other distinguished entries are excluded from the selection.
Thus the auxiliary selection problem remaining in a block is
\[
\begin{array}{c|cc}
 &\text{residual support size}&\text{residual row quota}\\\hline
 \text{unchosen block}&a-1&b\\
 \text{chosen block}&a-2&b-1.
\end{array}
\]
In a chosen block, relabel the designated row as $0$ and write $R_t$ for its
residual auxiliary selection. If $c$ is the distinguished colour and $q$
its partner, the full rows are $\{c\}\cup R_0$ and $\{q\}\cup R_t$ for $t\ne0$.
This prescribed pair contributes one to $c$ and $m-1$ to $q$; the
remaining selections must contribute residue zero to partners and
$\pm1$ to unmatched auxiliary colours.

\begin{lemma}[Matched selection]\label{lem:matched}
If every component of the residual block--auxiliary incidence graph contains an
even number of auxiliary colours that are not partners, then there is a row
selection giving integer total one to every distinguished cycle and total
$\pm1$ modulo $m$ to every auxiliary colour. It respects all row quotas.
Within each block, the selected and complementary row families can be made
intersection-connected whenever their members have size at least two.
\end{lemma}
\begin{proof}
In the residual incidence graph, prescribe odd degree at unmatched auxiliary colours and even
degree at partners and block vertices. The component hypothesis gives a parity join
on a spanning forest. Pair at blocks and orient the resulting auxiliary
multigraph as in Proposition~\ref{prop:pinned-selection}. Adding hub edges only at the
odd-degree vertices gives divergence zero at partners and $\pm1$ at unmatched auxiliary colours.
The oriented-pair row patterns therefore contribute zero modulo $m$ at partners and
$\pm1$ at unmatched auxiliary colours. The reserved slot contributes $m-1$ at each partner and one at
each distinguished cycle.

The residual quotas have enough room for the pair construction. An
unanchored residual block has support size $a-1$ and quota $b$; an anchored
one has size $a-2$ and quota $b-1$. These are nearest-half quotas. With $s$
disjoint pairs, $s$ is at most the smaller of the quota and its complement.
Complete the selection with unpaired auxiliary colours, chosen on every row.

In an unanchored block, assign the first two residual pairs, when present,
to exceptional rows $0,1$, and all remaining pairs to row $0$. In an anchored block,
cyclically relabel the rows so that the distinguished selection occurs on
row $0$. Assign the first two residual pairs, when present, to rows $1,2$,
and all remaining pairs to row $1$. Row $3$ has the same residual selection as the distinguished
row. This uses $m\ge4$.

In an unanchored block, the selected supports are connected through a fixed
selected column or the two exceptional rows; all complementary supports
contain the distinguished colour. In an anchored block, every selected
support except the distinguished row contains the partner. If the selected
supports have size at least two, the distinguished row shares its other
$b-1$ entries with row
$3$, joining the family. For the complements, every row except the
distinguished row contains the distinguished colour. The distinguished row
shares its other $a-b-1$ entries with row $3$ when the complementary supports
have size at least two. Hence both row families are intersection-connected.
\end{proof}

\subsection{The residual components for the initial factorization}
Apply Lemma~\ref{lem:matched} to the $x$-direction of~\eqref{eq:seed-widths}.
The blocks are indexed by $(h,w)$ and their row index is $x$. In the order
$F_0,F_1,F_{2,0},F_{2,1}$, use the four source vertices in~\eqref{eq:anchors}.
Their block indices are $(1,0),(2,0),(0,0),(1,3)$ and their row indices are
$0,1,0,0$. For $p\ge4$, choose partners by Table~\ref{tab:mates}; the cases
$p=2,3$ are treated in Lemmas~\ref{lem:entry-seven} and~\ref{lem:entry-nine}.

\begin{table}[htbp]
\centering
\caption{Auxiliary partners for the first $x$-split.}\label{tab:mates}
\begin{tabular}{ccccc}
\toprule
& $F_0$ & $F_1$ & $F_{2,0}$ & $F_{2,1}$\\
\midrule
$p\ge4$ even & $B_0$ & $B_1$ & $A_0$ & $A_1$\\
$p\ge5$ odd & $B_0$ & $B_1$ & $A_0$ & $A_2$\\
\bottomrule
\end{tabular}
\end{table}

\begin{lemma}\label{lem:entry-components}
Assume $p\ge4$. All partners in Table~\ref{tab:mates} are eligible and distinct. After deleting the
specified partner incidence in each of the four blocks, the residual auxiliary
components are
\begin{equation}\label{eq:residual-even}
 \{A_0\},\quad\{B_0\},\quad
 \{A_1,\ldots,A_{p-1}\},\quad\{B_1,\ldots,B_{p-1}\}
 \quad(p\text{ even}),
\end{equation}
and
\begin{equation}\label{eq:residual-odd}
 \{A_0\},\quad\{\text{all other auxiliary colours}\}
 \quad(p\text{ odd}).
\end{equation}
Every component contains an even number of unmatched auxiliary colours.
\end{lemma}
\begin{proof}
The pair table and the four three-colour rows show that the partners are
eligible. We compute the components after the four prescribed deletions.

First let $p=2k$, $k\ge2$. At height $1$, the generic $x$ support contains
$A_1,A_3,\ldots,A_{2k-1}$. The $w=1$ support attaches $A_2$ through $A_1$.
At $w=0$, deletion of $B_0$ leaves a support containing $A_3$ and all even-indexed
$A$ colours from $A_4$ onwards. Thus all $A_1,\ldots,A_{p-1}$ remain connected.
The generic row $w=2$ remains after deleting $A_1$ at $w=3$, including when
$m=4$, and preserves these connections.

At height $2$, the generic $x$ support contains $B_1,B_3,B_5,\ldots$.
The $w=1$ row attaches $B_4$ when it exists. At $w=0$, after deleting $B_1$, the
remaining support contains $B_2$ and $B_5$, together with the even-indexed $B$
colours from $B_6$ onwards. If $k\ge3$, this connects all $B_1,\ldots,B_{p-1}$.
If $k=2$, the height-$3$ support $\{B_2,B_3\}$ supplies the missing connection.
The height-$0$ rows attach only to this same $B$ component. The only $x$
incidences of $A_0$ and $B_0$ were precisely their deleted ones, so they are
isolated. No row joins the two displayed large components. This
proves~\eqref{eq:residual-even}.

Now let $p=2k+1$, $k\ge2$. The generic height-$1$ $x$ support consists of $B_0$
and the even-indexed $A$ colours. The $w=1$ row attaches $A_3$, and the $w=0$
row, after deleting $B_0$, contains $A_4$ and the other odd-indexed $A$ colours,
including $A_1$. Thus every $A_1,\ldots,A_{p-1}$ joins $B_0$.
The generic $w=2$ row preserves these connections after the deletion of
$A_2$ at $w=3$.
At height $2$, the $w=0$ row after deletion of $B_1$ contains both $B_0$ and
$B_4$. The generic row contains $B_1$ and all positive even-indexed $B$ colours;
the $w=1$ row supplies $B_3$, and the $w=0$ row supplies the remaining odd
indices. Hence all $B$ colours join the same component. The sole incidence of
$A_0$ has been deleted, proving~\eqref{eq:residual-odd}.

In the even case the isolated vertices are partners, and each large component has
$p-2$ unmatched auxiliary colours, an even number. In the odd case the large component has $2p-4$
unmatched auxiliary colours, again even.
\end{proof}

\begin{proposition}[Initial admissible marked factorization]\label{prop:entry}
For every even $m\ge4$ and $p\ge4$, one balanced $x$-split of the initial factorization can be
performed so that the result is an admissible marked factorization with distinguished colour set
$\{F_0,F_1,F_2\}$ and marked set $U_m\times\{0\}$. The recolouring from
Section~\ref{sec:anchored}, transported to this set, makes every colour Hamilton.
\end{proposition}
\begin{proof}
Lemma~\ref{lem:entry-components} verifies the component hypothesis of
Lemma~\ref{lem:matched}. Thus each distinguished cycle receives total voltage one,
at its designated point of $Q_m$, and every auxiliary cycle receives a unit
voltage. The selected row sets give a coordinate split by
Lemma~\ref{lem:physical}.

The initial factorization's other directions have even full incidence by
Lemma~\ref{lem:seed-incidence}; they retain it. The full intersection connectivity established in
Lemma~\ref{lem:matched} proves even incidence in each new child of multiplicity at least two by the proof of Lemma~\ref{lem:inherit}. A child of multiplicity one requires no further split; the connectivity
condition applies to children of larger multiplicity.

Every distinguished selected source is outside $U_m$. On each distinguished cycle meeting
$U_m$, its sole nonzero voltage therefore lies in the interior of one gap.
Theorem~\ref{thm:onegap} places every nonzero new fibre point in one enlarged
gap. The marked set lies on the zero section, and its relative recolouring is
valid and retains its cycle counts by Proposition~\ref{prop:transport}.
Cycle-consistency and the condition of at most one distinguished colour per direction follow from the coordinate
lift. These are all the conditions in Definition~\ref{def:collar}.

On the initial factorization, replacing the three distinguished factors by $H_0,H_1,H_2$ leaves the
auxiliary Hamilton factors unchanged and makes every colour Hamilton. The
transported recolouring has the same number of cycles in every colour.
\end{proof}

\subsection{Initial constructions in dimensions seven and nine}
For $p\ge4$, the residual components were computed in
Lemma~\ref{lem:entry-components}. For $p=2,3$, the following direct selections
give the first split. In particular, the partners for $p=3$ are chosen
together with the additional adjustment on $B_1,B_2$ in
Lemma~\ref{lem:entry-nine}; the generic odd-$p$ row of
Table~\ref{tab:mates} applies only to $p\ge5$.
The distinguished sources are the four points of $Q_m$,
in the order $F_0,F_1,F_{2,0},F_{2,1}$.

\begin{lemma}[Initial marked factorization in dimension seven]\label{lem:entry-seven}
For $p=2$ assign the four anchor blocks the partners
$(B_0,B_1,A_0,A_1)$. Split the $x$ direction with selected multiplicity one. In each
anchor block select the distinguished colour at its anchor and its partner at all other
sources. In every other block select its unique auxiliary colour using $x$.
This supplies the admissible marked factorization and Hamilton recolouring of
Proposition~\ref{prop:entry} in dimension seven.
\end{lemma}
\begin{proof}
Each partner is the unique auxiliary colour using $x$ in its anchor block. Every distinguished
cycle receives total voltage exactly one, and every auxiliary colour receives residue
$-1$: all other contributions are multiples of $m$. Both $x$ children have multiplicity
one. The untouched $y,w$ directions retain the componentwise evenness from
Lemma~\ref{lem:small-seed}. The sole distinguished voltage on each cycle is outside
$U_m$, so the single-interval lift creates an admissible marked factorization and transports the anchored
Hamilton recolouring exactly as in Proposition~\ref{prop:entry}.
\end{proof}

\begin{lemma}[Initial marked factorization in dimension nine]\label{lem:entry-nine}
For $p=3$ use the partners $(A_1,B_0,A_0,A_2)$ and split $x$ with selected
multiplicity one.
Use the same rule in the four anchor blocks. In the separate block
$(h,w)=(0,1)$ select $B_2$ at $x=0$ and $B_1$ elsewhere. In every other
nonanchor block choose a fixed auxiliary colour using $x$. This supplies the admissible marked factorization
and Hamilton recolouring of Proposition~\ref{prop:entry} in dimension nine.
\end{lemma}
\begin{proof}
The partners are present by the explicit table for the initial factorization at $p=3$. They each receive an
$m-1$ contribution. The extra block has exactly two auxiliary colours using $x$, namely $B_1,B_2$,
and supplies residues $-1,+1$. All other contributions are multiples of $m$.
Thus the four distinguished totals are one, and the auxiliary residues, in order
$A_0,A_1,A_2,B_0,B_1,B_2$, are $(-1,-1,-1,-1,-1,+1)$.

The selected child has multiplicity one. Complementary rows of any nonanchor block
share their old distinguished colour. In an anchor block the two complementary patterns
are $\{\text{partner},q\}$ and $\{\text{distinguished},q\}$, with $q$ the other auxiliary
colour, and hence intersect. Full complementary intersection connectivity follows. Incidence
parity is inherited, and the single-interval argument applies at the four sources
outside $U_m$.
\end{proof}

\subsection{Completion in every odd dimension at least seven}
\begin{theorem}\label{thm:oddconstruction}
For every even $m\ge4$ and every odd $d\ge7$, $D_d(m)$ has a Hamilton decomposition.
\end{theorem}
\begin{proof}
Write $d=2p+3$, where $p\ge2$. Construct the initial factorization
$T_m(p-\lfloor p/2\rfloor+1,\lfloor p/2\rfloor+1,p+1)$ and the anchored
Hamilton recolouring of Section~\ref{sec:anchored}. At $p=2,3$ use
Lemmas~\ref{lem:small-seed}, \ref{lem:entry-seven}, and \ref{lem:entry-nine};
at $p\ge4$ use Proposition~\ref{prop:entry}. In every case one split gives
an admissible marked factorization in four coordinates.

Apply Theorem~\ref{thm:closure} for exactly $d-4$ further balanced splits.
All final multiplicities are one. Each old cycle lifts to a single cycle,
so the final factorization still has $d-1$ Hamilton factors and one factor
consisting of two cycles. Apply the anchored recolouring on the iterated
zero-section copy of $U_m$. It gives the same cycle counts as the initial
Hamilton recolouring: one cycle in each colour.
The source directions still partition the positive basis arcs, and each factor
has length $m^d$. This proves the claim.
\end{proof}

\begin{corollary}\label{cor:support}
In the construction of Theorem~\ref{thm:oddconstruction}, the final recolouring changes only
three colours and exactly $m^3-m^2-4m+7$ sources. This number is independent of
$d$. Each source has one image in the zero section of every subsequent lift.
\end{corollary}
\begin{proof}
The initial support is given by Proposition~\ref{prop:patch-size}.
Proposition~\ref{prop:transport} takes one copy of this support at each lift.
Distinct distinguished directions remain distinct under coordinate
splitting, so each changed row remains changed after lifting.
\end{proof}

\medskip
\noindent\textbf{Example: the construction of $D_7(4)$.}
Take $p=2$. The initial factorization of $T_4(2,2,3)$ has six Hamilton
factors of length $64$ and one factor with two cycles of length $32$.
The four prescribed agreement vertices are
\[
 Q_4=\{(0,1,0),(1,1,0),(0,0,0),(0,2,3)\}.
\]
Lemma~\ref{lem:entry-seven} assigns the partners $(B_0,B_1,A_0,A_1)$ and splits
$x$ first. Every distinguished cycle receives its unique selected source
in $Q_4$, disjoint from $U_4$. The first split creates the marked-gap
condition; the remaining splits preserve it. One choice of multiplicities is
\[
\begin{aligned}
 (2,2,3)&\longrightarrow(1,1,2,3)
          \longrightarrow(1,1,1,1,3)\\
        &\longrightarrow(1,1,1,1,2,1)
          \longrightarrow(1,1,1,1,1,1,1).
\end{aligned}
\]
Here we split $x$, then $y$, then $w$, and finally the multiplicity-two child
of $w$. Each old cycle lifts to one cycle, so before recolouring there are
still eight cycles: six of length $4^7$ and two of length $4^7/2$.
Under the successive coordinate identifications~\eqref{eq:sum-sheet}, the
copy of the recolouring support is
\[
 \{(x,0,y,0,w,0,0):(x,y,w)\in U_4\},\qquad |U_4|=39.
\]
The same three-colour recolouring as in dimension three acts at these
39 sources. Since its first-return permutations are unchanged, the four
distinguished cycles become three Hamilton cycles. The four auxiliary Hamilton factors remain fixed, giving seven
Hamilton cycles. Thus all four coordinate splits preserve the
reconnection carried by these 39 sources.

\section{A five-dimensional construction by transversal cycle joining}\label{sec:d5}
The preceding three-coordinate construction requires a positive
auxiliary multiplicity in every direction. Two auxiliary colours cannot
supply these three multiplicities, so the five-dimensional case uses
transversal cycle joining.

Five affine-section modifications organize the return cycles as cosets
of subgroups of $\Zm^4$. For $m\ge6$, two colours become Hamilton;
each remaining colour has $m$ cycles, all meeting its final planar
support in two points. Their first returns are involutions, reducing
the last three joining problems to permutations on $2m$ points.
At $m=4$, a four-row initial assignment uses the same affine sections
and planar rule, with first returns computed on eight points.

\subsection{Transversals under a preceding composition}
We use additive notation for finite abelian groups. A transversal in the current
layer pulls back through the preceding composition to a transversal of the old
return cycles. The induced action on the quotient determines which cycles are
joined.

\begin{lemma}[Transversal cycle joining under conjugation]\label{lem:chronological}
Let $X=H\oplus W$ be a finite abelian group, let $e\in W$ have order $m$, and let
$S$ be a permutation whose cycles are exactly the cosets of $H$. Suppose that a
permutation $P$ satisfies
\begin{equation}\label{eq:prefix-quotient}
 P(x)+H=x+b+H\qquad(x\in X)
\end{equation}
for a fixed $b\in X$. On an affine coset $C=z+W$, define
\[
 J(y)=\begin{cases}y+e,&y\in C,\\y,&y\notin C.\end{cases}
 \qquad r=P^{-1}JP.
\]
Then the cycles of $Sr$ are exactly the cosets of $H+\langle e\rangle$.
In particular, each new cycle joins exactly $m$ old cycles and has length $m|H|$.
\end{lemma}
\begin{proof}
The set $C$ meets each $H$-coset once. Equation~\eqref{eq:prefix-quotient} implies
that $P$ sends every $H$-coset bijectively to a translate of that coset. Thus
$A=P^{-1}C$ also meets every old $S$-cycle exactly once. The first return of $S$
on $A$ is therefore the identity.

Since $e\in W$, the map $J$ is a permutation supported on $C$. Each of its cycles
on $C$ is an $e$-line of length $m$. Its $m$ points belong to different $H$-cosets:
$\langle e\rangle\cap H=\{0\}$. Under $P^{-1}$ such a line gives $m$ marked
representatives of the $H$-cosets inside one coset of $H+\langle e\rangle$.
Indeed, \eqref{eq:prefix-quotient} gives $r(x)+H=x+e+H$ on $A$.
Source reconnection now has return permutation $r$ on $A$. Each of its $m$-cycles
concatenates the complete old $H$-cycles through its representatives. Every old cycle is represented, proving the claimed vertex sets and lengths.
\end{proof}

To apply the lemma to a layer direction assignment, let $P$ be the composition of the maps before
the layer being changed. A change from the old layer $L$ to $LJ$ changes the full
return from $R$ to
\begin{equation}\label{eq:chronological-identity}
 R'=R(P^{-1}JP).
\end{equation}
The relevant transversal is consequently $P^{-1}C$. If all earlier corrections
have displacements in $H$, their preceding composition has the form
$P(x)=x+b+h_x$, with $h_x\in H$. Its action on $X/H$ is translation by $b$,
which is precisely~\eqref{eq:prefix-quotient}; the displacement $h_x$ within a
coset may depend on $x$.

\begin{remark}[Quotient conjugacy]\label{rem:quotient-conjugacy}
More generally, suppose $P$ permutes the $H$-cosets and induces a permutation
$\overline P$ of $X/H$. Then $P^{-1}C$ is still a transversal. On the labels of
the old cycles, the restriction of the relative permutation to this transversal induces
\[
 \overline P^{-1}\,\tr_{\overline e}\,\overline P,
\]
where $\tr_{\overline e}$ is translation by the class of $e$.
Source reconnection joins the old cycles in each orbit of this quotient permutation.
Every new cycle still joins $m$ old cycles; when $\overline P$ is a translation,
the joined vertex sets are exactly the cosets of $H+\langle e\rangle$.
Thus the quotient permutation determines the joining order, and a
translation on the quotient preserves the description by cosets.
\end{remark}

\subsection{Layer assignments and affine sections}
In the next two subsections, let $m\ge6$ be even. Use height-section coordinates
$X=\Zm^4$, with standard vectors $u_0,u_1,u_2,u_3$ and $u_4=0$; put
$\alpha_{ab}=u_b-u_a$. The height-coordinate representation for the physical torus is
\[
 (x,h)\longmapsto
 (x_0,x_1,x_2,x_3,h-x_0-x_1-x_2-x_3).
\]
A step in these coordinates $(x,h)\mapsto(x+u_j,h+1)$ is precisely a positive step in
physical direction $j$, including $j=4$.
The initial direction of colour $c\in\{0,\ldots,4\}$ at height $t$ is
\begin{equation}\label{eq:d5-neutral}
 \rho_t(c)=\begin{cases}c+t\pmod5,&0\le t\le4,\\c,&5\le t<m.\end{cases}
\end{equation}
The corresponding return to the initial height layer is translation by
\[
 \lambda_c=\sum_{j=0}^3u_j-5u_c\quad\hbox{in }\Zm^4.
\]
Every $\lambda_c$ is primitive over $\Z$, so the initial return cycles are the
cosets of the cyclic module it generates, each of length $m$.

In this section, permutations displayed as direction cycles act on
the \emph{direction symbols}, with the colour labels fixed. Thus, at
a source $x$ in the indicated affine section, a direction cycle $\pi$
replaces the current direction $d_{c,t}(x)$ of colour $c$ by
$\pi(d_{c,t}(x))$. For example, $(0\ 2\ 1)$ replaces directions
$0,2,1$ by $2,1,0$, respectively. Each section condition is evaluated
at the source before taking the step. The changes in
Table~\ref{tab:d5-cylinders} are applied in the displayed order;
changes at the same height have disjoint direction symbols and commute.
Write $d^{\mathrm{pre}}_{c,t}(x)$ for the resulting direction after
these five changes and before the final planar recolouring.
\begin{table}[htbp]
\centering\small
\caption{The five intermediate five-dimensional modifications.}\label{tab:d5-cylinders}
\begin{tabular}{cccl}
\toprule
$P$&height&direction cycle&affine section $C_P$\\\midrule
0&0&$(0\ 2\ 1)$&$x_0+x_1+x_2+x_3=0$\\
1&0&$(3\ 4)$&$x_2=1$\\
2&1&$(1\ 0\ 4)$&$x_2=x_3=0$\\
3&1&$(2\ 3)$&$x_1=0,\quad x_0+x_2+x_3=-1$\\
4&2&$(2\ 4)$&$x_0=x_1=x_3=0$\\\bottomrule
\end{tabular}
\end{table}

The direction modules of these affine sections, with their ordered bases, are
\begin{equation}\label{eq:d5-W}
\begin{aligned}
 W_0&=\langle\alpha_{01},\alpha_{02},\alpha_{03}\rangle,&
 W_1&=\langle\alpha_{01},\alpha_{03},\alpha_{34}\rangle,\\
 W_2&=\langle\alpha_{01},\alpha_{04}\rangle,&
 W_3&=\langle\alpha_{02},\alpha_{23}\rangle,\\
 W_4&=\langle\alpha_{24}\rangle.
\end{aligned}
\end{equation}
Each $C_P$ is an affine coset of $W_P$. Affected colours and their displacements
are listed in Table~\ref{tab:d5-integral}. In particular every displacement lies
in $W_P$. After subtracting its initial translation, each modified layer map is
translation by that displacement on $C_P$ and the identity elsewhere. Hence it
is bijective. At a common height the two sets of direction symbols are disjoint, so the
direction assignment at each source is a permutation, including where the
affine sections overlap.

For each colour, consider the modifications in increasing height order.
Before each modification, let $H^-$ have the ordered basis consisting of
$\lambda_c$ followed by that colour's earlier displacements. Define $Q_{c,P}=[H^-\ W_P]$ using
\eqref{eq:d5-W}. The following table gives the determinant over $\Z$ and the
coordinates $v$ such that $Q_{c,P}^{-1}e_{c,P}=(0,v)$.
\begin{table}[htbp]
\centering\small
\caption{Integral complement identities for the successive five-dimensional cycle joinings.
The $W_P$-coordinate vectors are primitive.}\label{tab:d5-integral}
\begin{tabular}{cccrcl}
\toprule
$c$&$P$&$\operatorname{rank}H^-$&$\det Q_{c,P}$&$e_{c,P}$&$v$\\\midrule
0&0&1&$-1$&$\alpha_{02}$&$(0,1,0)$\\
1&0&1&$-1$&$\alpha_{10}$&$(-1,0,0)$\\
2&0&1&$-1$&$\alpha_{21}$&$(1,-1,0)$\\
3&1&1&$-1$&$\alpha_{34}$&$(0,0,1)$\\
4&1&1&$-1$&$\alpha_{43}$&$(0,0,-1)$\\
0&2&2&$-1$&$\alpha_{10}$&$(-1,0)$\\
3&2&2&$-1$&$\alpha_{41}$&$(1,-1)$\\
4&2&2&$1$&$\alpha_{04}$&$(0,1)$\\
1&3&2&$1$&$\alpha_{23}$&$(0,1)$\\
2&3&2&$-1$&$\alpha_{32}$&$(0,-1)$\\
0&4&3&$1$&$\alpha_{24}$&$(1)$\\
2&4&3&$-1$&$\alpha_{42}$&$(-1)$\\\bottomrule
\end{tabular}
\end{table}
The displayed integer matrices are unimodular. Their inverse matrices therefore
have integer entries, and reduction modulo every $m$ gives
$X=H^-\oplus W_P$. The primitive integer coordinate vector $v$ extends to a basis of
$\Z^{\operatorname{rank}W_P}$. Consequently, $e_{c,P}$ has order $m$, and
adjoining it to $H^-$ increases the rank of the free direct summand by one. These integral identities hold after reduction modulo every modulus.

\begin{proposition}[Intermediate five-dimensional factors]\label{prop:d5-preterminal}
Before the final modification, colours $0,2$ have one height-layer return cycle of
length $m^4$. Colours $1,3,4$ each have $m$ cycles of length $m^3$, respectively
the cosets of $\ker f_c$, where
\begin{equation}\label{eq:d5-characters}
 f_1=(2,2,3,3),\qquad f_3=(-1,0,1,0),\qquad f_4=(0,-1,1,0).
\end{equation}
These coset descriptions hold also on the section immediately before height $2$.
\end{proposition}
\begin{proof}
Induct over Table~\ref{tab:d5-integral} for each colour. Before each
modification, the preceding composition has displacement equal to a fixed
initial translation plus a sum of earlier colour displacements. Its quotient modulo
$H^-$ is consequently a fixed translation. Lemma~\ref{lem:chronological},
\eqref{eq:chronological-identity}, and the integer complement table show that the
return cycles after each modification are the cosets of the enlarged
module.

For colours $0,2$ the resulting module has rank four and is $X$. The other three
ordered generating sets are
\[
 (\lambda_1,\alpha_{10},\alpha_{23}),\quad
 (\lambda_3,\alpha_{34},\alpha_{41}),\quad
 (\lambda_4,\alpha_{43},\alpha_{04}).
\]
They are free direct summands of size $m^3$ and are annihilated by the respective
characters in \eqref{eq:d5-characters}. Each character is surjective: for $f_1$
one may use $u_2-u_0$, and the other two have a coefficient $1$ or $-1$.
Their kernels also have size $m^3$, proving equality. Finally, the preceding composition up to
height $2$ is a translation on each of these quotient groups. Conjugating the
return to that section therefore preserves the family of cosets.
\end{proof}

\subsection{A planar recolouring and its first returns}
On $\Zm^2$ put $a_0=(1,0)$, $a_1=(0,1)$, $a_2=(0,0)$. A word
$b_0b_1b_2$ denotes the permutation $i\mapsto b_i$ of $\{0,1,2\}$;
thus $012$ is the identity and $021$ exchanges $1$ and $2$.
Define $\omega_m(q)$ by the word $012$ except at the points in
Table~\ref{tab:d5-omega}.
The rule is defined already for $m=4$; empty ranges are omitted.
\begin{table}[htbp]
\centering\small
\caption{The final planar direction row.}\label{tab:d5-omega}
\begin{tabular}{cl@{\qquad}cl}
\toprule
$q$&$\omega_m(q)$&$q$&$\omega_m(q)$\\\midrule
$(1,0)$&$021$&$(3,0)$&$102$\\
$(1,1)$&$120$&$(2,1)$&$201$\\
$(0,2)$&$201$&$(2,2)$&$120$\\
$(t,2),\ 3\le t<m$&$210$&$(0,3)$&$120$\\
$(1,3)$&$201$&$(1,t),\ 4\le t<m$&$021$\\
$(m+3-t,t),\ 4\le t<m$&$102$&&\\\bottomrule
\end{tabular}
\end{table}

For $i=0,1,2$ define
\begin{equation}\label{eq:d5-localj}
 j_i(q)=q-a_i+a_{\omega_m(q)(i)},\qquad
 B_i=\{q:\omega_m(q)(i)\ne i\}.
\end{equation}
The relevant quotient functionals on the plane are
$\ell_0(a,b)=a$, $\ell_1(a,b)=b$, and $\ell_2(a,b)=a+b$.

The following cyclic lists give the local permutations and their first returns. Put
\[
 A=(0,2),\quad B=(0,3),\quad C=(2,1),\quad
 X_t=(t,2),\quad Y_t=(1,t),\quad Z_t=(t,3-t),
\]
where coordinates are read modulo $m$. In lists below, a range is taken in the
indicated order and is omitted if empty. Substitution in
\eqref{eq:d5-localj} gives the following full cycles; outside them $j_i$ is the
identity:
\begin{align}
 j_0={}&(A, X_{m-1},\ldots,X_2,Y_3,B,Z_{m-1},\ldots,Z_3,C,Y_1),\label{eq:d5-j0}\\
 j_1={}&(A,Y_1,Y_0,Y_{m-1},\ldots,Y_3,X_2,C,Z_3,\ldots,Z_{m-1},B),\label{eq:d5-j1}\\
 j_2={}&(A,B,Y_3,\ldots,Y_{m-1},Y_0,Y_1,C,X_2,\ldots,X_{m-1}).\label{eq:d5-j2}
\end{align}
In particular these are permutations, and $|B_i|=2m$.
Each $\ell_i$-fibre meets $B_i$ twice. The permutations swapping those two points
are precisely
\begin{align}
 N_0={}&(A\ B)(Y_1\ Y_3)(X_2\ C)
          \prod_{t=3}^{m-1}(X_t\ Z_t),\label{eq:d5-N0}\\
 N_1={}&(Y_0\ Z_3)(Y_1\ C)(A\ X_2)(Y_3\ B)
          \prod_{t=4}^{m-1}(Y_t\ Z_{m+3-t}),\label{eq:d5-N1}\\
 N_2={}&(A\ Y_1)(B\ C)(Y_0\ X_{m-1})
          \prod_{t=2}^{m-2}(X_t\ Y_{t+1}).\label{eq:d5-N2}
\end{align}
All transpositions within each product are disjoint.

\begin{lemma}[The three final first-return cycles]\label{lem:d5-endpoint}
For every even $m\ge4$, each $N_i j_i$ is a single cycle on $B_i$.
\end{lemma}
\begin{proof}
For an indexed tuple $(V_t,W_t)_{t=r,r+2,\ldots,s}$, concatenate the indicated
pairs. Directly composing \eqref{eq:d5-j0}--\eqref{eq:d5-j2} with
\eqref{eq:d5-N0}--\eqref{eq:d5-N2} gives the following complete cyclic lists:
\begin{align}
 N_0j_0={}&\bigl(A,
 (Z_t,X_{t-1})_{t=m-1,m-3,\ldots,3},Y_1,B,\notag\\[-2pt]
 &\hspace{27mm}(X_t,Z_{t-1})_{t=m-1,m-3,\ldots,5},X_3,C,Y_3\bigr),\label{eq:d5-end0}\\
 N_1j_1={}&\bigl(A,C,Y_0,
 (Z_t,Y_{m+2-t})_{t=4,6,\ldots,m-2},B,X_2,Y_1,Z_3,\notag\\[-2pt]
 &\hspace{27mm}(Y_{m+4-t},Z_t)_{t=5,7,\ldots,m-1},Y_3\bigr),\label{eq:d5-end1}\\
 N_2j_2={}&\bigl(A,C,(Y_t,X_t)_{t=3,5,\ldots,m-1},Y_1,B,X_2,\notag\\[-2pt]
 &\hspace{27mm}(Y_t,X_t)_{t=4,6,\ldots,m-2},Y_0\bigr).\label{eq:d5-end2}
\end{align}
For example, in the first list $A$ maps to $Z_{m-1}$, and
$Z_t\mapsto X_{t-1}$ for odd $t\ge3$. For odd $t\ge5$ one has
$X_{t-1}\mapsto Z_{t-2}$; the end $X_2$ maps to $Y_1$ and then to $B$.
The other half is checked by the even-indexed transpositions. The same
substitution verifies both indexed ranges and their endpoints in the other two
lists. Each list contains exactly once the $2m$ distinct points in the
corresponding support, proving cyclicity. The empty-range convention also gives
these identities at $m=4$.
\end{proof}

Place this rule on the final plane
\[
 \Theta=\{(b,-a-b,0,a):a,b\in\Zm\},\qquad
 \phi(a,b)=(b,-a-b,0,a),
\]
at height $2$. On $x=\phi(q)$ permute the direction symbols
$(d_0,d_1,d_2)=(3,0,1)$ by $d_i\mapsto d_{\omega_m(q)(i)}$; act identically off
$\Theta$. These directions are used by colours $(1,3,4)$ when $m\ge6$.
The pair $(2\ 4)$ already in height $2$ is disjoint from this triple.
Here $i=0,1,2$ indexes the listed direction symbols; the corresponding
colour labels are $1,3,4$, respectively. The row remains a permutation
of the direction symbols, and the relative map for the corresponding
colour on the plane is $j_i$, since
\begin{equation}\label{eq:d5-head}
 \phi(j_i(q))+u_{d_i}=\phi(q)+u_{d_{\omega_m(q)(i)}}.
\end{equation}
The explicit cycles above prove that this layer is bijective.

\begin{theorem}[The five-dimensional construction for even moduli at least six]\label{thm:d5large}
For every even $m\ge6$, the complete direction assignment just defined is a Hamilton
decomposition of $D_5(m)$.
\end{theorem}
\begin{proof}
The direction assignments partition the arcs at each source, and every
colour maps each height layer bijectively onto the next. The final
three-colour recolouring leaves the Hamilton returns of colours $0,2$
unchanged. On $\Theta$ the three residual characters restrict to
\[
 f_1(\phi(a,b))=a,\qquad
 f_3(\phi(a,b))=-b,\qquad f_4(\phi(a,b))=a+b.
\]
Thus each of the $m$ old return cycles of a residual colour contains exactly two
points of its final support. Its first return on that support is exactly the
appropriate $N_i$, independently of the lengths of the intervening paths.
Every old cycle meets the support.

Start the height return immediately before height $2$. By
\eqref{eq:d5-head}, the new return is the old return composed on the right with
$j_i$, extended to the identity off the plane. Lemma~\ref{lem:surgery} and
Lemma~\ref{lem:d5-endpoint} show that it has one cycle of length $m^4$.
Taking account of the height coordinate gives a Hamilton cycle of length $m^5$
for every colour.
\end{proof}

\subsection{The construction at modulus four}
The initial assignment~\eqref{eq:d5-neutral} uses the five heights
$0,\ldots,4$. At $m=4$ we instead prescribe four rows, keeping the same
affine sections and planar rule. This changes the intermediate cycle
characters and the colours affected by the last recolouring.
Replace~\eqref{eq:d5-neutral} by
\begin{equation}\label{eq:d5four-schedule}
\begin{array}{c|ccccc}
 t\backslash c&0&1&2&3&4\\\hline
 0&0&1&2&3&4\\
 1&1&3&4&2&0\\
 2&4&2&3&0&1\\
 3&3&4&0&1&2.
\end{array}
\end{equation}
Every row is a permutation. The initial return vectors, in colour order, are
\[
 \mu_0=(1,1,0,1),\quad\mu_1=(0,1,1,1),\quad\mu_2=(1,0,1,1),\quad
 \mu_3=(1,1,1,1),\quad\mu_4=(1,1,1,0).
\]
Apply the transversal cycle-joining lemma over $\mathbb Z_4$ with $H^-$ generated by $\mu_c$
and earlier displacements. Table~\ref{tab:d5four-cert} gives the complement
data; the listed $v$ satisfy $Q^{-1}e=(0,v)$ modulo $4$. The determinants are
odd, hence units modulo $4$, and each $v$ contains a unit coordinate. In
particular, the determinants $\pm1,\pm3$ give invertible complements over
$\mathbb Z_4$, as required by the cycle-joining lemma.
\begin{table}[htbp]
\centering\small
\caption{The modulus-four complement identities.}\label{tab:d5four-cert}
\begin{tabular}{cccrcl}
\toprule
$c$&$P$&$\operatorname{rank}H^-$&$\det Q$&$e$&$v\pmod4$\\\midrule
0&0&1&3&$\alpha_{02}$&$(0,1,0)$\\
1&0&1&3&$\alpha_{10}$&$(3,0,0)$\\
2&0&1&3&$\alpha_{21}$&$(1,3,0)$\\
3&1&1&$-1$&$\alpha_{34}$&$(0,0,1)$\\
4&1&1&$-1$&$\alpha_{43}$&$(0,0,3)$\\
0&2&2&$-1$&$\alpha_{10}$&$(3,0)$\\
2&2&2&1&$\alpha_{41}$&$(1,3)$\\
4&2&2&1&$\alpha_{04}$&$(0,1)$\\
1&3&2&$-3$&$\alpha_{32}$&$(0,3)$\\
3&3&2&1&$\alpha_{23}$&$(0,1)$\\
0&4&3&1&$\alpha_{42}$&$(3)$\\
1&4&3&1&$\alpha_{24}$&$(1)$\\\bottomrule
\end{tabular}
\end{table}
Consequently colours $0,1$ have one intermediate $256$-cycle; colours $2,3,4$
have four $64$-cycles each. On the section before height $2$, their cycles are
respectively the cosets of the kernels of
\[
 g_2(x)=x_0-x_3,\qquad g_3(x)=x_0-x_1,\qquad g_4(x)=-x_1+x_2.
\]
This follows from the same successive cycle-joining argument and the final free rank-three
modules; the displayed surjective characters annihilate their generators.

At modulus four, the final recolouring acts on colours $(2,3,4)$; the
corresponding colours for larger even moduli were $(1,3,4)$.
Write $U_c$ for the intermediate height-layer return beginning immediately before height
$2$. The restrictions on the final plane are
\[
 g_2\phi(a,b)=b-a,\quad g_3\phi(a,b)=a+2b,\quad g_4\phi(a,b)=a+b.
\]
Use the notation $A,B,C,X_t,Y_t,Z_t$ from the preceding subsection, with $m=4$.
The local maps $j_i$ still have the cycles
\eqref{eq:d5-j0}--\eqref{eq:d5-j2}; only the old first returns change.
Their permutations on the eight-point supports are
\begin{align}
 \widetilde N_0&=(A\ Y_3)(B\ C\ X_3)(Y_1\ X_2)(Z_3),\label{eq:d5four-N0}\\
 \widetilde N_1&=(A\ C)(B\ X_2)(Y_1\ Y_3\ Z_3)(Y_0),\label{eq:d5four-N1}\\
 \widetilde N_2&=(A\ Y_1)(B\ C)(Y_0\ X_3)(Y_3\ X_2).\label{eq:d5four-N2}
\end{align}
Singletons represent fixed points. The quotient characters determine
the pairs and singletons; the three-point intersection for each of
colours $2,3$ has its cyclic order determined by the four identities
\begin{equation}\label{eq:d5four-iterate}
\begin{array}{c|c|c|c}
 c&q&k&U_c^k(\phi(q))\\\hline
 2&B&38&\phi(C)\\
 2&B&59&\phi(X_3)\\
 3&Y_1&30&\phi(Y_3)\\
 3&Y_1&47&\phi(Z_3).
\end{array}
\end{equation}
Starting at $x_0=\phi(q)$, iterate
$x_{s+1}=x_s+u_{d^{\mathrm{pre}}_{c,(2+s)\bmod4}(x_s)}$ for $4k$ steps.
This uses the four specified layer maps before the final recolouring.
The old return cycles have length $64$, and the displayed marked points
occur in the order $0<38<59<64$ or $0<30<47<64$, fixing the orientation
of each three-point return. The supplementary calculations in
Appendix~\ref{app:computation} reproduce these iterates and the printed
return permutations.

Direct composition with the local $j_i$ gives
\begin{align}
 \widetilde N_0j_0&=(A,B,Z_3,X_3,Y_1,Y_3,C,X_2),\notag\\
 \widetilde N_1j_1&=(A,Y_3,B,C,Y_1,Y_0,Z_3,X_2),\label{eq:d5four-end}\\
 \widetilde N_2j_2&=(A,C,Y_3,X_3,Y_1,B,X_2,Y_0).\notag
\end{align}
Every old cycle meets the relevant support: their intersection cardinalities
are $1,2,2,3$ for each of colours $2,3$, and $2,2,2,2$ for colour $4$.
Source reconnection therefore joins all four old cycles for each residual colour.

\begin{theorem}[The five-dimensional decomposition]\label{thm:d5}
For every even $m\ge4$, $D_5(m)$ has a Hamilton decomposition.
\end{theorem}
\begin{proof}
For $m\ge6$ use Theorem~\ref{thm:d5large}. At $m=4$, all layer maps of
\eqref{eq:d5four-schedule} with the prescribed modifications are bijective by the
support-invariance and planar-permutation arguments already given. The rows are
permutations of the direction symbols. Colours $0,1$ are intermediate Hamilton factors. For the remaining three,
\eqref{eq:d5four-end} and the fact that every old cycle meets the support
give one $256$-cycle
on the height section. Including height gives five $1024$-cycles, partitioning
all positive arcs of $D_5(4)$.
\end{proof}

\subsection{Completion of the proof}
\begin{theorem}[Even moduli]\label{thm:even}
For every even $m\ge4$ and every $d\ge2$, $D_d(m)$ has a Hamilton decomposition.
\end{theorem}
\begin{proof}
Even dimensions follow from Theorem~\ref{thm:even-dim}. Odd dimensions at
least seven follow from Theorem~\ref{thm:oddconstruction}. Dimension three
is Proposition~\ref{prop:anchor}, and dimension five is Theorem~\ref{thm:d5}.
These cases exhaust $d\ge2$.
\end{proof}

\begin{proof}[Proof of Theorem~\ref{thm:main}]
For odd $m$, use Theorem~\ref{thm:oddmod}; for even $m\ge4$, use
Theorem~\ref{thm:even}. Both proofs are contained in this paper.
\end{proof}

\section{Consequences of the all-moduli theorem}\label{sec:consequences}
We apply Theorem~\ref{thm:main} to products of spanning cycles,
changes of generating basis, and spanning paths obtained by cutting Hamilton
cycles. Throughout this section, $m\ge3$ and $d\ge2$ unless otherwise stated.

\subsection{Cartesian products and powers}
The Cartesian product $G_1\square\cdots\square G_d$ has vertex set
$\prod_i V(G_i)$; an arc changes exactly one coordinate along an arc of the
corresponding factor. Parallel arcs, when present, retain their labels.

\begin{proposition}[Universal transfer principle]\label{prop:universal-transfer}
Fix $n\ge2$ and $d\ge2$. The following statements are equivalent.
\begin{enumerate}
\item $D_d(n)$ has a Hamilton decomposition.
\item For every $r\ge1$ and every collection of loopless $n$-vertex
      digraphs $G_1,\ldots,G_d$, each decomposable into $r$ directed
      Hamilton cycles, the product $G_1\square\cdots\square G_d$
      decomposes into $dr$ directed Hamilton cycles.
\end{enumerate}
Parallel arcs, if present, are separately labelled.
\end{proposition}
\begin{proof}
For (i)$\Rightarrow$(ii), write the decomposition of $G_i$ as
$C_{i,1},\ldots,C_{i,r}$. For each $j$, the spanning subdigraph
\[
 H_j=C_{1,j}\square\cdots\square C_{d,j}
\]
is isomorphic to $D_d(n)$ via the cyclic order in each coordinate.
Every product arc changes a unique coordinate along an arc belonging
to a unique factor $C_{i,j}$, so the $H_j$ partition all product arcs.
Apply (i) separately to each $H_j$. For (ii)$\Rightarrow$(i), take
$r=1$ and $G_i=\vec C_n$ for every $i$.
\end{proof}

Corollary~\ref{cor:cartesian-products} now follows from
Theorem~\ref{thm:main}; the assertion for the first Cartesian power is
the given decomposition. In this reduction, the modulus becomes the
common order of the factors, so the all-moduli range yields the product
result for every $n\ge3$.

\begin{example}[Two distinct arithmetic boundaries]\label{ex:two-factor-boundary}
The graph $\vec C_3\square\vec C_4$ has no Hamilton cycle. Such a cycle
would use $3a$ and $4b$ arcs in the respective directions, with $a,b\ge1$
and $3a+4b=12$. Reduction modulo three forces $3\mid b$, which is
incompatible with $a\ge1$.
In contrast, $\vec C_3\square\vec C_6$ is Hamiltonian but has no Hamilton
decomposition. By the Trotter--Erd\H{o}s criterion one cycle exists when
$a+b=\gcd(3,6)$ and $\gcd(a,3)=\gcd(b,6)=1$; $(a,b)=(2,1)$ works.
Keating's decomposition criterion would require instead
$\gcd(ab,18)=1$. The only positive pairs summing to three are $(1,2)$
and $(2,1)$, and neither works. See~\cite[Theorems~2.1 and~2.3]{DMM22}
for these classical criteria. The first example exhibits the role of
factor orders, and the second distinguishes the arithmetic of
Hamiltonicity from that of Hamilton decomposition.
\end{example}

\begin{remark}[Chinese remainders and the generating set]\label{rem:crt-graph}
For coprime $a,b\ge2$, the Chinese remainder isomorphism identifies
$\Z_{ab}^d$ with $\Z_a^d\times\Z_b^d$, but a positive coordinate step
becomes $(e_i,e_i)$. Consequently
\[
 D_d(ab)\cong\Cay\bigl(\Z_a^d\times\Z_b^d,
                       \{(e_i,e_i):1\le i\le d\}\bigr).
\]
By contrast, $D_d(a)\square D_d(b)$ has $2d$ directions,
$(e_i,0)$ and $(0,e_i)$. Chinese remainders turn a coordinate step into
a simultaneous step in the two factors; Cartesian multiplication
permits the two steps separately. A lifting argument must therefore
respect the generating set as well as the group decomposition.
\end{remark}

\subsection{Height-stretched skew tori}
Preserving the first-return permutations also permits the number of
height layers to increase. This gives a family on nonhomocyclic abelian
groups by inserting layers whose total action is the identity.

\begin{lemma}[Inserting layers with identity return]\label{lem:neutral-height}
Let $H$ be a finite abelian group of exponent $e$, let
$\Delta=\{\delta_1,\ldots,\delta_r\}\subseteq H$, and let $q\ge2$.
If
\[
 \Gamma_q(H,\Delta)=\Cay\bigl(\Z_q\times H,
                    \{(1,\delta):\delta\in\Delta\}\bigr)
\]
has a Hamilton decomposition, then so does $\Gamma_{q+te}(H,\Delta)$
for every integer $t\ge0$. The decomposition can be chosen to have the
same colourwise first returns on the height-zero section.
\end{lemma}
\begin{proof}
The given factors specify bijections $U_{c,j}:H\to H$ on the $q$
successive height layers. At every source the differences
$U_{c,j}(x)-x$ use $\Delta$ exactly once. For each new block of $e$
layers, assign colour $c$ the constant translation $x\mapsto x+\delta_c$
on every layer. This is a Latin assignment, and its $e$-fold composition
is the identity. Appending $t$ such blocks leaves every colour return
unchanged. Lemma~\ref{lem:layer-return} completes the proof.
\end{proof}

\begin{corollary}[Height-stretched directed tori]\label{cor:skew-height}
For every $m\ge3$, $d\ge2$, and $a\ge1$, the digraph
\begin{equation}\label{eq:skew-height}
 \Cay\left(\Z_{am}\times\Z_m^{d-1},
       \{(1,0),(1,e_1),\ldots,(1,e_{d-1})\}\right)
\end{equation}
has a decomposition into $d$ directed Hamilton cycles.
Its directed girth is $am$. For $a>1$, it is not isomorphic to
$\vec C_{am}\square\vec C_m^{\square(d-1)}$.
\end{corollary}
\begin{proof}
At $a=1$, the coordinate map
\[
 (x_0,x_1,\ldots,x_{d-1})\longmapsto
 (x_0+x_1+\cdots+x_{d-1},x_1,\ldots,x_{d-1})
\]
identifies $D_d(m)$ with \eqref{eq:skew-height}. Apply
Theorem~\ref{thm:main} and Lemma~\ref{lem:neutral-height}, since
$\Z_m^{d-1}$ has exponent $m$.
Every arc raises the height by one modulo $am$, so every directed cycle
has length at least $am$. Repeating $(1,0)$ gives a cycle of that length.
The ordinary unequal-side Cartesian product has directed girth $m$,
which proves the last assertion.
\end{proof}

The extension includes the even fibres $m=2^b$ with $b\ge2$.
The added layers are chosen in blocks whose return is the identity,
which preserves each colour's original return permutation. This
condition is the source of the height extension.

\subsection{Generating bases and Cayley digraphs}
Decomposing connection data into independent sets and then into
cycle-product subgraphs also occurs in the undirected setting; see the
Paley-graph proof of Alspach, Bryant and Dyer~\cite{ABD12}, in particular
its independent-set partition and basis-generated subgraphs. Here the
building block is a directed positive-basis torus over a residue ring,
so the transported factors retain the chosen generator orientations.

A basis below is a module basis over $R=\Z/m\Z$. If the connection data form a
multiset, each occurrence of a generator specifies a separately labelled arc
at every source.

\begin{corollary}[A partition into generating bases]\label{cor:basis-partition}
Let $m\ge3$, $d\ge2$, and $k\ge1$. Suppose a connection multiset $S$ on
$R^d$ has a partition
\[
 S=B_1\mathbin{\dot\cup}\cdots\mathbin{\dot\cup}B_k
\]
in which each $B_j$ is an $R$-basis of $R^d$. Then $\Cay(R^d,S)$ has a
decomposition into $kd$ directed Hamilton cycles. In particular, every choice
of a generating basis gives a Hamilton decomposition with $d$ factors.
\end{corollary}
\begin{proof}
Write $B_j=\{b_{j,1},\ldots,b_{j,d}\}$. The map
\[
 \varphi_j:R^d\longrightarrow R^d,\qquad
 \varphi_j(x)=\sum_{i=1}^{d}x_i b_{j,i}
\]
is an automorphism and sends $x\to x+e_i$ to
$\varphi_j(x)\to\varphi_j(x)+b_{j,i}$. It therefore transports the
decomposition of Theorem~\ref{thm:main} to a decomposition of
$\Cay(R^d,B_j)$. The partition of the generator occurrences partitions all
labelled arcs, so the $k$ decompositions combine as asserted.
\end{proof}

For instance, let $U_i\subseteq R^\times$ have the same positive cardinality
$k$ for $1\le i\le d$. Enumerate $U_i=\{u_{i,1},\ldots,u_{i,k}\}$. Each set
$\{u_{1,j}e_1,\ldots,u_{d,j}e_d\}$ is a basis, so
\[
 \Cay\bigl(R^d,\{u e_i:1\le i\le d,\ u\in U_i\}\bigr)
\]
decomposes into $kd$ directed Hamilton cycles. Taking one sign per coordinate
gives every uniform choice of coordinate orientations. Taking
$U_i=\{1,-1\}$ gives $2d$ Hamilton factors in the bidirected torus; these may
be chosen as $d$ cycles and their reversals. Forgetting the directions in the
positive decomposition gives $d$ undirected Hamilton cycles in
$C_m^{\square d}$, with the additional property that all cycles respect the
prescribed positive orientation.

For a prime-power modulus, the basis-partition hypothesis has an exact
rank formulation over the residue field.

\begin{corollary}[A prime-power rank criterion]\label{cor:prime-power-rank}
Let $m=p^a\ge3$, where $p$ is prime and $a\ge1$, and let $d\ge2$ and $k\ge1$.
Let $S$ be a multiset of $kd$ elements of $(\Z/p^a\Z)^d$, and write $\bar s$
for reduction modulo $p$. If
\begin{equation}\label{eq:rank-density}
 \bigl|\{s\in S:\bar s\in W\}\bigr|\le k\dim_{\mathbb F_p}W
 \qquad\text{for every subspace }W\le\mathbb F_p^d,
\end{equation}
where multiplicities are counted, then $\Cay((\Z/p^a\Z)^d,S)$ decomposes into
$kd$ directed Hamilton cycles. Condition~\eqref{eq:rank-density} is equivalent
to the existence of a partition of $S$ into $k$ module bases.
\end{corollary}
\begin{proof}
Regard the occurrences in $S$ as the ground set of the vector matroid
represented by the vectors $\bar s$. For every submultiset $A\subseteq S$,
put $W=\operatorname{span}_{\mathbb F_p}\{\bar s:s\in A\}$. The hypothesis gives
\[
 |A|\le\bigl|\{s\in S:\bar s\in W\}\bigr|
       \le k\dim_{\mathbb F_p}W.
\]
Edmonds' matroid partition theorem~\cite[Theorem~1]{Edmonds65} partitions $S$
into $k$ sets whose reductions are linearly independent. Each has at most $d$
elements, and their total size is $kd$, so every set has exactly $d$ elements.
The determinant of each corresponding matrix is nonzero modulo $p$, hence is
a unit modulo $p^a$. Each set is therefore a module basis, and
Corollary~\ref{cor:basis-partition} applies. Conversely, the reduction of a
module basis is a basis over $\mathbb F_p$, and contains at most $\dim W$
vectors in any subspace $W$. Summing over the $k$ bases proves
\eqref{eq:rank-density}.
\end{proof}

This criterion includes the even moduli $2^a$ with $a\ge2$. For example, on
$(\Z/4\Z)^2$ the four generators
\[
 (1,0),\ (0,1),\ (1,1),\ (1,2)
\]
partition into the two displayed consecutive pairs, both of determinant one.
Their Cayley digraph consequently decomposes into four Hamilton cycles.
For $p=2$, this applies to modules over $\Z/2^a\Z$ with $a\ge2$,
whose residue field is $\mathbb F_2$. The rank condition characterizes
partitions into module bases and thereby gives a sufficient condition
for Hamilton decomposition.

At a general composite modulus the generator occurrences must admit
\emph{one} partition which is a basis partition modulo every prime
dividing $m$. Independently finding partitions at the separate primes
is weaker, even when every individual generator is primitive.

\begin{example}[Incompatible primewise basis partitions]\label{ex:crt-bases}
In $(\Z/30\Z)^2$ take the four generators
\[
 s_1=(1,0),\quad s_2=(15,16),\quad
 s_3=(10,21),\quad s_4=(6,25).
\]
Their reductions, in this order, are
\[
\begin{array}{c|cccc}
 &s_1&s_2&s_3&s_4\\\hline
\bmod 2&e_1&e_1&e_2&e_2\\
\bmod 3&e_1&e_2&e_1&e_2\\
\bmod 5&e_1&e_2&e_2&e_1.
\end{array}
\]
Every row admits a partition into two bases of $\mathbb F_p^2$ and
satisfies the rank-density inequalities with $k=2$. Nevertheless, in
pair order $12,13,14,23,24,34$, the integer determinants are
\[
 16,\ 21,\ 25,\ 155,\ 279,\ 124,
\]
whose greatest common divisors with $30$ are respectively
$2,3,5,5,3,2$. No pair is a module basis over $\Z/30\Z$, so a common
basis partition is impossible. The full set still generates the module: $s_1=e_1$ and
$4s_2-3s_3=(30,1)=e_2$ modulo $30$. Thus this generating set lies beyond the common-basis method, even
though all its primewise reductions admit the required basis partitions.
\end{example}

\subsection{Spanning paths and the sharp arc-deletion threshold}
For a digraph $G$, let $\delta_G^-(v)$ and $\delta_G^+(v)$ denote the sets of
arcs entering and leaving a vertex $v$, respectively.

\begin{corollary}[Rooted paths and the arc-deletion threshold]\label{cor:packing-resilience}
Let $G=D_d(m)$, where $d\ge2$ and $m\ge3$. For every vertex $v$, the arcs of
$G-\delta_G^-(v)$ partition into $d$ directed Hamilton paths, all starting at
$v$. Similarly, $G-\delta_G^+(v)$ decomposes into $d$ directed Hamilton paths,
all ending at $v$.

If $F\subseteq E(G)$ and $|F|<d$, then $G-F$ contains at least $d-|F|$
pairwise arc-disjoint directed Hamilton cycles. The minimum number of arcs
whose deletion destroys all directed Hamilton cycles is exactly $d$.
\end{corollary}
\begin{proof}
Fix a Hamilton decomposition $C_1,\ldots,C_d$ of $G$. Deleting from each $C_i$
its unique arc entering $v$ produces a Hamilton path starting at $v$. Those
$d$ deleted arcs are exactly $\delta_G^-(v)$, which proves the first partition.
Deleting instead the unique arc leaving $v$ proves the second.

Every arc of $F$ belongs to exactly one $C_i$, so at most $|F|$ of the chosen
cycles meet $F$. At least $d-|F|$ cycles survive intact. Thus at least $d$
arc deletions are needed to destroy Hamiltonicity. Deleting all $d$ arcs
leaving a single vertex destroys every directed Hamilton cycle, proving
sharpness.
\end{proof}

For the equal-side family with $m\ge3$, the corollary gives the optimal
value $d$ in the path-packing problem of Darijani, Miraftab and Witte
Morris~\cite[Question~1.2]{DMM22}, with any prescribed common initial
vertex. This value is optimal even when initial vertices are unrestricted:
writing $N=m^d$, we have $(d+1)(N-1)>dN$, whereas the digraph has $dN$ arcs.
The sharp deletion threshold concerns the total number of removed arcs.
Both the rooted-path and deletion conclusions hold, with $d$ replaced by
$r$, in every digraph decomposed into $r$ directed Hamilton cycles,
including the Cartesian and Cayley families above.

\subsection{Common transversals of cycle decompositions}
Two cycle decompositions can be compared through a set of arcs meeting every
cycle in each. Throughout this comparison the cycles are simple and
parallel arcs retain their labels.

\begin{proposition}[A minimum common cycle transversal]\label{prop:common-cuts}
Let a finite loopless digraph on $N$ vertices have a Hamilton decomposition
$\mathcal H$ into $r$ cycles and another decomposition $\mathcal F$ of the
same arc set into $c$ directed cycles. There is a set $E$ of exactly $c$
arcs which contains one arc from every member of $\mathcal F$ and meets
every member of $\mathcal H$. No smaller set can meet every cycle of
both decompositions.
After removing $E$, both decompositions consist of exactly $c$ directed
paths, allowing paths of length zero, and have the same uncoloured
incoming and outgoing boundary defects.
\end{proposition}
\begin{proof}
Form the bipartite overlap graph between $\mathcal H$ and $\mathcal F$,
joining two cycles when they share an arc. A subset $\mathcal A\subseteq
\mathcal H$ contains $N|\mathcal A|$ distinct arcs. Each neighbouring
cycle of $\mathcal F$ contains at most $N$ arcs, so
\[
 N|\mathcal A|\le N|N(\mathcal A)|.
\]
Hall's theorem~\cite{Hall35} provides a matching saturating $\mathcal H$.
From each matched intersection choose an arc, and choose any one arc
from each unmatched cycle of $\mathcal F$. These $c$ arcs meet every
cycle on both sides. Since the cycles of $\mathcal F$ are arc-disjoint,
$c$ is also a lower bound.
Cutting $k\ge1$ arcs of a cycle gives $k$ paths (possibly of length zero),
hence there are $c$
paths on each side. The same physical arcs were removed, so the missing
incoming and outgoing incidences agree.
\end{proof}

\begin{corollary}[Cycle excess and common transversals]\label{cor:excess-cuts}
Suppose a digraph has a decomposition into $r$ Hamilton cycles and also
an $r$-colour permutation factorization with total cycle excess $E$.
The minimum number of common cut arcs meeting every cycle in both
factorizations is exactly $r+E$.
\end{corollary}
\begin{proof}
The second factorization has $r+E$ cycles. Apply
Proposition~\ref{prop:common-cuts}.
\end{proof}

The common transversal determines equal uncoloured incoming and
outgoing boundary incidences. A recolouring additionally prescribes
which colour follows each resulting path; in the relative splitting
theorem that information is supplied by the first-return permutations.

\section{Quantitative features of the construction}\label{sec:constructive}
Coordinate splitting controls two different quantities: the number of
local direction profiles and the number of sources changed by the final
recolouring. The first admits a polynomial bound in $m,d$; the second
remains independent of $d$ because relative lifting carries the
recolouring to one copy of its original support.

\subsection{The number of local direction profiles}
At a fixed stage, group fibre blocks by their complete direction profile and the
label of the exceptional $F_2$ cycle. Every other colour is Hamilton, so its cycle label is fixed. A single
profile type may represent many blocks. In the parity join one representative block of each type
suffices: duplicate supports have the same column connectivity. Give all other
blocks a constant auxiliary-only selection; such a selection exists because
there is at most one distinguished column and $\lfloor a/2\rfloor\le a-1$.

\begin{proposition}\label{prop:types}
The construction in Theorem~\ref{thm:oddconstruction} uses at most
\begin{equation}\label{eq:type-bound}
 4m^2+3d(d-4)
\end{equation}
distinct fibre-block profile types after the first split and at every
subsequent stage. A profile consists of the direction of every colour
and the label of the exceptional $F_2$ cycle.
\end{proposition}
\begin{proof}
The selected parity-join incidences form a forest on the $d+1$ cycle-column
vertices and the used representative block vertices. Every used block vertex
has positive even degree, hence degree at least two. If there are $B$ such
vertices, the forest inequality gives $2B\le B+(d+1)-1$, so $B\le d$.
A used representative block has at most three row patterns: the two
exceptional patterns and the ordinary pattern. Together with the constant pattern on other blocks of its
type, it creates at most four new types from one old type. Thus a split increases
the number of types by at most $3d$. After the first split, each of the $m^2$
initial blocks has at most four row types. There are $d-4$ further splits, proving
\eqref{eq:type-bound}.
\end{proof}

The bound concerns distinct profiles; reconstructing a successor also
requires their locations. These are specified by the split directions,
the coordinates and selections of the representative blocks, and the
assignment of the remaining blocks to profile types. Together with
\eqref{eq:sum-sheet}, this information determines the successor at a
given vertex. Bounding the cost of locating and evaluating this profile
when $m$ and the coordinates are given in binary is a further question.

\subsection{Cycle excess and localized recolouring}
For even $m$ and odd $d$, cycle excess one is the least value allowed
by the canonical sign class. The initial factorization realizes this
value together with the stronger, componentwise incidence condition.
The lower bound below compares the support needed to join these
factors with that required by the canonical coordinate factorization.

\begin{proposition}[A support lower bound]\label{prop:support-lower}
Let $S$ and $T$ be permutations of the same finite set, let $T$ be a
single cycle, and suppose $S$ has $k>1$ cycles. Then
\[
 |\{x:Sx\ne Tx\}|\ge k.
\]
The bound is attained if no restriction on the permitted successor arcs
is imposed.
\end{proposition}
\begin{proof}
Put $W=\{x:Sx\ne Tx\}$ and $r=S^{-1}T$.
Then $r$ is supported on $W$ and is a product of at most $|W|-1$
transpositions. A transposition changes the number of cycles by one,
so $k-1\le |W|-1$. For equality, choose one point in each $S$-cycle
and let $r$ cyclically permute these $k$ points. Source reconnection
joins the $k$ cycles into one and changes exactly those $k$ sources.
\end{proof}

A direct Hamilton recolouring of the canonical coordinate factorization
requires at least $m^{d-1}$ changed sources in each colour. The
factorization used here has cycle excess one before lifting. Relative
lifting transports its Hamilton recolouring to the same three-dimensional
section at every stage.

\section{Further questions}\label{sec:extensions}
Theorem~\ref{thm:main} gives Hamilton decompositions at every cycle
length $m\ge3$. Extending the method calls for arc assignments that
both produce cyclic returns and preserve intervals for later lifts.
The questions below concern the size of the required recolourings and
the hypotheses under which such assignments can be repeated.

\subsection{The size of the recolouring}
Let $s(m)$ be the minimum number of sources in a three-colour Hamilton
recolouring of Table~\ref{tab:nearcore} that preserves the four agreements
at $Q_m$. The anchored construction gives
\[
 s(m)\le m^3-m^2-4m+7.
\]
The four agreements leave the sources required by the first lift
outside the recolouring support.

\begin{question}\label{ques:support}
What is the order of growth of $s(m)$ over even $m$? In particular, is
$s(m)=O(m^2)$?
\end{question}

The cubic upper bound in $m$ comes from a recolouring on one
three-dimensional section, whose copy is retained through every later
split. A smaller initial support would therefore improve the bound in
all higher odd dimensions covered by the same construction.

\subsection{Selection and iteration}
Appendix~\ref{app:coupled} characterizes one marked split by integral cut
inequalities and the gap congruences
\[
 n_{c,C,j}\equiv-K_Cq_{C,j}\pmod m,
 \qquad K_C\in\mathbb Z_m^\times,
 \qquad q_{C,j}\ge0,\quad\sum_jq_{C,j}=m-1.
\]
The half quotas and repeated supports used here ensure that suitable
choices can be made again after each split.

Degree congruences and connectivity are also central to modulo-factor
theory. Hasanvand~\cite{Hasanvand25} obtains bounded-degree factors under
edge-connectivity hypotheses; in a bipartite graph with even degrees on
one side, his results also allow exact half-degrees on that side and
prescribed residues on the other. Complementary modulo factors can
further be required to contain prescribed numbers of edge-disjoint
spanning trees~\cite{Hasanvand22TC}. Here the repeated supports permit
exact row quotas, prescribed exclusions, and intersection connectivity
within each block. The latter yields the incidence parity needed for
the next split. Extending this combination leads to the following
question.

\begin{question}\label{ques:quotas}\label{ques:iterative-selection}
Which conditions on general row quotas, forbidden entries, and marked
intervals ensure the existence of such selections while preserving the
same conditions under coordinate splitting?
\end{question}

Conditions expressed through the current incidence system would give
an induction on $\sum_i(a_i-1)$, extending the relative splitting
theorem beyond repeated supports and half quotas.

\subsection{Marked intervals and allowable voltages}
First returns determine the effect of a recolouring, while the positions
of the voltages govern subsequent lifts. In the single-interval
construction, if $G$ is the chosen open interval of a cycle $C$ and
$K=C\setminus G$, then
\[
 \widehat G=(G\times\{0\})\,\dot\cup\,
 \bigl(C\times(\mathbb Z_m\setminus\{0\})\bigr),
 \qquad \widehat K=K\times\{0\}.
\]
The complementary path remains on the zero section; the chosen
interval contains every newly introduced nonzero fibre point. For
several intervals, the integer counts also record positional information:
$m$ selections and zero selections give equal residues and different
fibre traversals.

\begin{question}\label{ques:return-times}\label{ques:resource-state}
When voltages are supported on several marked intervals, what positional
information yields a class closed under relative splitting? Can its
size be bounded in terms of the initial marked factorization?
\end{question}

\subsection{Uniform initial factorizations}
The three-coordinate initial family supplies the cases of odd $d\ge7$;
the five-dimensional case uses transversal cycle joining. A common
initial family would need a prescribed Hamilton recolouring and a
first split satisfying Theorem~\ref{thm:closure}.

\begin{question}\label{ques:entry-unification}
Is there a family of initial factorizations in a bounded number of
coordinates, including the five-colour case, with prescribed boundary
agreements and a first split satisfying Theorem~\ref{thm:closure}?
\end{question}

Directional counts constrain this problem. Every Hamilton cycle on
$\mathbb Z_m^n$ uses each direction a positive multiple of $m$ times.
Moreover, a Hamilton factorization with constant directional
multiplicities has at least $n$ factors. The auxiliary number in our
three-dimensional superposition is even, hence at least four, which
accounts for its starting dimension seven. A uniform family including five colours would therefore require a
different assignment of the auxiliary arcs among coordinate directions.

\subsection{Other generating sets and directed products}
At a prime power, Corollary~\ref{cor:prime-power-rank} reduces basis
partition to rank inequalities. At a general modulus a common partition
must work at all prime divisors, as Example~\ref{ex:crt-bases} illustrates.
Unequal-side products additionally require compatible coordinate maps
and incidence selections; the two-factor criteria of
Trotter--Erd\H{o}s and Keating~\cite{TE78,Keating85} give the first
arithmetic constraints.

\begin{question}\label{ques:other-products}
Which prescribed generating sets admit quotient or covering
constructions whose arc assignments and cyclic returns can be preserved
under repeated relative lifting? Which explicit families extend the
common-basis and height-stretched constructions?
\end{question}

The all-moduli theorem provides a Hamilton decomposition at every
common cycle length $m\ge3$. Its product consequence applies at every
common vertex order at least three to Hamilton-decomposable factors of
the same degree, while the common-basis construction supplies directed
Hamilton decompositions for the corresponding Cayley families.

The relative splitting theorem gives a further principle: for an
admissible marked factorization, it preserves the cycle-count effect
of every valid distinguished recolouring on the marked section while
retaining the intervals required for further refinement. A low-dimensional
recolouring can therefore complete the decomposition after all coordinate
multiplicities have been split. Broader incidence conditions would
extend the range of these relative constructions; smaller completing
recolourings with prescribed agreements would improve their support
bounds throughout the higher-dimensional family.

\appendix
\section{The return calculation for the cyclic-star layer}\label{app:star}
We prove the cyclicity assertion used in Lemma~\ref{lem:star}.
The calculation passes from a height-layer return to a $2m$-point section, and
then to the horizontal-axis points. The resulting first returns have explicit
cyclic orders, while their return times account for all points of each section.

Let $Q$ be the $2m$ points in~\eqref{eq:star-support}, let $P$ be that displayed
cycle, extended by the identity outside $Q$, and let $\nu$ be the first return of
a translation $\tr_v$ to $Q$. By Lemma~\ref{lem:surgery}, the first return of
$\tr_vP$ on $Q$ is
\[
 \Phi=\nu P.
\]
In both cases of~\eqref{eq:star-voltage}, every $\tr_v$-cycle meets $Q$.
For $v=(1,-2)$ the translation cycles are distinguished by $b+2a$:
$Y_1,\ldots,Y_{m-1}$ meet every nonzero value, and $X_{m/2}$ meets value zero.
For $v=(1,3)$, when $3\mid m$, use the invariant $b-3a$, the same $Y$ points,
and $X_{m/3}$. The first coordinate of $v$ is one, so every translation cycle has
length $m$. It is consequently enough to show that $\Phi$ has one cycle.

The following rule determines every entry in the return tables. For $q=(a,b)\in Q$ and $v=(1,\lambda)$, where
$\lambda=-2$ or $3$, choose the least $t\in\{1,\ldots,m\}$ for which one of
\begin{equation}\label{eq:hitting-tests}
\begin{split}
 &b+\lambda t=0,\quad a+t\ne0;\\
 &a+t=0,\quad b+\lambda t\ne0;\\
 &(a+t,b+\lambda t)=(1,-1);\\
 &(a+t,b+\lambda t)=(-1,1)
\end{split}
\end{equation}
holds in $\Zm$. Then $\nu(q)=q+tv$. To compute $\Phi$, first move to the next
point of~\eqref{eq:star-support}, and then apply this hitting rule. The tables
record its further first return to
\[
 \mathcal X=\{X_s:1\le s<m\}.
\]
The return times in the tables are measured in \emph{steps of $\Phi$ on $Q$}.
Their sums account for all $2m$ points of $Q$.

\subsection{The case \texorpdfstring{$3\nmid m$}{3 does not divide m}}
Write $m=2k$ and use $v=(1,-2)$. Substitution in~\eqref{eq:hitting-tests} gives
the following first return $\Psi$ on $\mathcal X$ and return time $\rho$:
\begin{equation}\label{eq:star-table-nondiv}
\begin{array}{c|c|c}
 s&\Psi(s)&\rho(s)\\\hline
 1\le s\le k-2&s+k+1&1\\
 s=k-1&k&1\\
 k\le s\le2k-3&s-k+2&3\\
 s=2k-2&k+1&4\\
 s=2k-1&1&3.
\end{array}
\end{equation}
Empty index ranges are omitted from the tables. For example, in
the first row $P(X_s)=X_{s+1}$ and the first translation hit is after $k$ steps,
namely $X_{s+k+1}$. For $s=k-1$, the prospective half-turn reaches the omitted
origin, so the first hit is the full return to $X_k$. For the remaining rows,
start with the vertical hit $t=m-s-1$ when $P(X_s)=X_{s+1}$, and use the two
corner tests when an axis endpoint is reached. Repeating the same rule gives
the last three rows, without an intermediate visit to $\mathcal X$.

On the labels $1,\ldots,n$, with $n=2k-1$ and residues represented by these
labels, put
\[
 \beta(s)=s+k+1\pmod n,\qquad \theta=(1\ k\ k+1).
\]
Table~\eqref{eq:star-table-nondiv} is exactly $\Psi=\theta\beta$.
Since $3\nmid k$, there are two cases. If $k\equiv2\pmod3$, then
$\gcd(n,k+1)=3$. The three points $1,k,k+1$ lie in the three different cycles of
$\beta$, and $\theta$ joins them into one cycle. If $k=3q+1$, then
$\gcd(n,k+1)=1$. On the $\beta$-cycle, the points $1,k,k+1$ occur in that order at
positions $0,2q,4q+1$. Cutting at these three positions shows that multiplying by
$\theta$ leaves one cycle. This argument applies for $q\ge1$; the boundary
$m=4$ belongs to the first case and gives $\Psi=(1\,2\,3)$ directly.

The sum of the return times in~\eqref{eq:star-table-nondiv} is $4k=2m$.
First-return paths from distinct points of $\mathcal X$ have disjoint interiors.
As $\Psi$ is one cycle, these paths exhaust $Q$ and prove that $\Phi$ is a single
$2m$-cycle. Lemma~\ref{lem:surgery} now implies that $\tr_{(1,-2)}P$ is a single
$m^2$-cycle.

\subsection{The case \texorpdfstring{$3\mid m$}{3 divides m}}
Write $m=6q$. Suppose first that $q\ge2$, and put
\[
 a=2q+1,\qquad n=6q-1=3a-4,\qquad
 \beta(s)=s+a\pmod n\quad(1\le s\le n).
\]
Define $\theta$ on $1,\ldots,a$ by
\begin{equation}\label{eq:tau-divisible}
\begin{aligned}
 \theta(1)&=a-1,&\theta(2)&=1,&\theta(3)&=a,\\
 \theta(s)&=s-2\quad(4\le s\le a-2),&&&\\[-2pt]
 \theta(a-1)&=a-2,&\theta(a)&=a-3,
\end{aligned}
\end{equation}
and fix all other labels. The images in~\eqref{eq:tau-divisible} are distinct
and exhaust $1,\ldots,a$, so $\theta$ is a permutation.

Using $v=(1,3)$ in~\eqref{eq:hitting-tests}, the first return of $\Phi$ on
$\mathcal X$ is
\begin{equation}\label{eq:star-divisible-return}
 \Psi=\theta\beta,
 \qquad
 \rho(s)=
 \begin{cases}
 1,&1\le s\le4q-1,\\
 5,&s=4q\text{ or }s=6q-2,\\
 4,&4q+1\le s\le6q-3,\\
 3,&s=6q-1.
 \end{cases}
\end{equation}
Here the horizontal hitting condition is $3t=-b$, the vertical hitting time is
$t=-a_0$ for a current point $(a_0,b)$, and the two corner conditions are the last
two lines of~\eqref{eq:hitting-tests}. Taking the least positive eligible time
after each successor in~\eqref{eq:star-support} gives the ranges in
\eqref{eq:star-divisible-return}. In particular, the two five-step returns are
separate for $q\ge2$. The same least-positive-time rule determines the boundary entries.

Since $a$ is odd,
\[
 \gcd(n,a)=\gcd(3a-4,a)=\gcd(4,a)=1,
\]
so $\beta$ is cyclic. Its first return to $I=\{1,\ldots,a\}$ is addition by four
modulo $a$: starting from $s\le a-4$ it takes three steps, and from $s>a-4$ it
takes two. Consequently the first return of $\Psi$ on $I$ is
$s\mapsto\theta(s+4)$, with residues represented in $I$. Its complete cyclic
order is
\begin{equation}\label{eq:small-star-cycle}
 1,3,\ldots,a-4,\quad
 a-3,a-1,a,\quad
 2,4,\ldots,a-5,\quad
 a-2,1.
\end{equation}
An ascending range whose upper endpoint is smaller than its lower endpoint is
empty. For $a=5$, for example, this reads $1,2,4,5,3,1$.
Substitution in~\eqref{eq:tau-divisible} verifies every transition in
\eqref{eq:small-star-cycle}; the list contains each member of $I$ exactly once.
Because $\theta$ is supported on $I$, every segment outside $I$ is an unchanged
$\beta$-segment. Thus every $\Psi$-cycle meets $I$, and $\Psi$ is one cycle on
$1,\ldots,n$.

The sum of the return times in~\eqref{eq:star-divisible-return} is
\[
 (4q-1)+10+4(2q-3)+3=12q=2m.
\]
It follows as before that $\Phi$ is one cycle on $Q$ and $\tr_{(1,3)}P$ is one
cycle on $\Zm^2$.

For the remaining case $m=6$, the same hitting rule gives, on the five $X$ labels,
\[
 \Psi=(1\,4\,3\,2\,5),\qquad
 (\rho(1),\ldots,\rho(5))=(1,1,1,5,4).
\]
The return times sum to $12=|Q|$, which completes the proof of
Lemma~\ref{lem:star} for every even $m\ge4$.

\section{Prescribed return times on one cycle}\label{app:capacity}
The zero-section formula of Proposition~\ref{prop:phase-order} also
allows prescribed increases in the return times. We first treat
arbitrary additive voltages and then the binary voltages used in
coordinate splitting, using the marked positions, base cycle length
$L$, old return times $\ell_j$, and fibre modulus $m$ of that proposition.

\begin{corollary}[Prescribing return-time increments]\label{cor:budget-realization}
In the notation above, let $q_j\ge0$ be integers with $\sum_jq_j=m-1$.
Suppose every gap with $q_j>0$ contains an unmarked vertex. There is an
additive lift with unit total, zero voltage on $U$, the original marked
order, and return times $\ell_j+q_jL$.
\end{corollary}
\begin{proof}
In each gap with $q_j>0$, put voltage $-q_j$ at one unmarked source,
and put zero elsewhere. The total is $-(m-1)=1$ modulo $m$.
Then $\kappa_j=\sum_{i<j}q_i$, which lies in $[0,m-1]$ and is
nondecreasing. Proposition~\ref{prop:phase-order} gives the assertion.
\end{proof}

\begin{proposition}[The binary capacity criterion for one cycle]\label{prop:binary-capacity}
For each open marked gap let $A_j$ be a set of eligible sources, of size
$c_j$. Fix integers $q_j\ge0$ with $\sum_jq_j=m-1$.
There exists a binary voltage supported on $\bigcup_jA_j$ with unit total,
the original marked order, and return times $\ell_j+q_jL$ if and only if there
is a unit $K\in\Zm$ such that
\begin{equation}\label{eq:binary-capacity}
 [-Kq_j]_m\le c_j\qquad\text{for every }j.
\end{equation}
\end{proposition}
\begin{proof}
Let $n_j$ be the selected count in gap $j$. If the prescribed return
increments hold, Proposition~\ref{prop:phase-order} gives
$n_j\equiv-Kq_j\pmod m$, including the last gap, since the total is $K$.
Thus $[-Kq_j]_m\le n_j\le c_j$.
Conversely, choose exactly $[-Kq_j]_m$ sources in each $A_j$.
Their total is $-K\sum_jq_j\equiv K\pmod m$.
The resulting phases are $\kappa_j=\sum_{i<j}q_i$; the order criterion
and the return-time formula now apply.
\end{proof}

For several colours, the selected sources must also satisfy the common
row sizes. The next appendix expresses that simultaneous choice through
cut inequalities; Theorem~\ref{thm:closure} supplies the additional
conditions for repeated splitting.

\section{Simultaneous marked binary selections}\label{app:coupled}
A simultaneous split couples the colourwise voltage conditions through
the number of selections allowed at each source. We express these
conditions as an integral flow problem with congruence constraints on
the column counts.

\subsection{Exact row quotas and prescribed column counts}
Let $X$ be a finite row set and $\Lambda$ a finite set of columns.
For each $x\in X$, let $A_x\subseteq\Lambda$ be its allowed entries and
$0\le b_x\le |A_x|$ its row quota. Put $B=\sum_xb_x$ and define
\[
 \rho(W)=\sum_{x\in X}\min\{b_x,|A_x\cap W|\}
 \qquad(W\subseteq\Lambda).
\]
Forbidden entries have already been removed from $A_x$.

\begin{lemma}[Integral cut criterion]\label{lem:quota-cuts}
There are selections $J_x\subseteq A_x$, $|J_x|=b_x$, with prescribed
nonnegative integer column totals $n_\lambda$ if and only if
\begin{equation}\label{eq:quota-cuts}
 \sum_{\lambda\in\Lambda}n_\lambda=B,
 \qquad
 \sum_{\lambda\in W}n_\lambda\le\rho(W)
 \quad\text{for every }W\subseteq\Lambda.
\end{equation}
For fixed targets $n_\lambda$, an integral max-flow computation either
constructs the selections or gives a violated cut inequality.
\end{lemma}
\begin{proof}
Necessity counts how many selected entries can lie in $W$.
For sufficiency form a network with arcs from the source to each row
of capacity $b_x$, allowed row--column arcs of capacity one, and arcs
from columns to the sink of capacities $n_\lambda$.
If $W$ is the set of columns on the sink side of a cut, minimization over
the positions of the row vertices gives cut capacity
\[
 \sum_{\lambda\notin W}n_\lambda+
 \sum_x\min\{b_x,|A_x\cap W|\}.
\]
Every cut has capacity at least $B$ by \eqref{eq:quota-cuts}.
The integral max-flow/min-cut theorem~\cite{FF56} gives a flow of value
$B$. All row and column margins are saturated, and the unit-capacity
entries of that integral flow give the selections.
\end{proof}

Singleton sets $W$ measure individual column capacities; larger sets
measure the capacity shared by several columns. For example, four
quota-one rows with
supports
\[
 \{A,B\},\quad\{C,D\},\quad\{C,D\},\quad\{C,D\}
\]
and counts $n_A=n_B=n_C=n_D=1$ satisfy the singleton inequalities and
the total-sum identity. The inequality for $W=\{A,B\}$ fails because
these two columns share a single eligible row.

\subsection{Prescribed marked returns}
Let $(S_c)$ factorize $T_m(a_1,\ldots,a_n)$, with $m\ge2$.
Fix a direction $i$ and $0<b<a_i$. Specify, for each colour, a marked
set $U_c$ and any additional forbidden sources for voltage one.
For every cycle $C$ meeting $U_c$, list its marks cyclically as
$u_{C,0},\ldots,u_{C,s_C-1}$, with old return lengths $\ell_{C,j}$.
Prescribe a unit $K_C\in\Z_m^\times$ for every colour-cycle and, on
marked cycles, integers
\[
 q_{C,j}\ge0,\qquad\sum_jq_{C,j}=m-1.
\]
The desired new return lengths are $\ell_{C,j}+q_{C,j}|C|$.

Make one column $\lambda=(c,C,j)$ for each open gap of every marked
cycle, and one column $\lambda=(c,C)$ for every unmarked cycle.
A source $x$ is eligible for that column exactly when colour $c$ uses
direction $i$ at $x$, $x$ lies in the indicated gap or unmarked cycle,
and the entry is not forbidden. In particular no marked source is
eligible for its own colour. These rules define $A_x\subseteq\Lambda$.
Take $b_x=b$ at every source. If some $|A_x|<b$, the split is infeasible.

\begin{theorem}[Simultaneous marked selection]\label{thm:marked-cuts}
A binary coordinate split with these allowed entries, unit totals $K_C$,
the original marked orders, and the prescribed return lengths exists
if and only if there are nonnegative integers $n_\lambda$ satisfying
\eqref{eq:quota-cuts} and
\begin{align}
 n_{c,C,j}&\equiv -K_Cq_{C,j}\pmod m
       &&\text{for each marked cycle and gap},\label{eq:gap-residues}\\
 n_{c,C}&\equiv K_C\pmod m
       &&\text{for each unmarked cycle}.\notag
\end{align}
Every old colour-cycle then lifts to one cycle.
\end{theorem}
\begin{proof}
Given a split, let $n_\lambda$ count its selected entries in each column. The row quotas imply \eqref{eq:quota-cuts}.
On a marked cycle, Proposition~\ref{prop:phase-order} gives
\eqref{eq:gap-residues}, including the last gap; the unmarked congruence
is the specified total voltage.
Conversely, realize the counts using Lemma~\ref{lem:quota-cuts} and put
$\delta_c(x)=1$ at the selected entries and zero elsewhere.
Each source has exactly $b$ selections in direction $i$, so
Lemma~\ref{lem:physical} gives a physical coordinate split.
The total on a marked cycle is
$-K_C\sum_jq_{C,j}\equiv K_C\pmod m$, and unmarked totals are $K_C$
by assumption. Lemma~\ref{lem:lift} therefore preserves each old cycle.
On a marked cycle, the partial-voltage phases are
\[
 \kappa_{C,j}=\sum_{h<j}q_{C,h},
\]
which lie in $[0,m-1]$ and are nondecreasing. The phase criterion gives
precisely the prescribed order and return lengths.
\end{proof}

\begin{proposition}[Selections with fixed margins]\label{prop:count-reconfiguration}
Any two realizations of the same feasible target vector in
Lemma~\ref{lem:quota-cuts} are connected by alternating-cycle flips
using only allowed entries. Every intermediate choice has the same
row quotas and column counts.
Consequently, for fixed targets in Theorem~\ref{thm:marked-cuts}, these
flips preserve the factorization, colourwise cycle counts, marked orders,
and prescribed return times throughout the reconfiguration.
\end{proposition}
\begin{proof}
In the symmetric difference of two realizations, orient an edge present
only in the first from its row to its column, and an edge present only
in the second from its column to its row. Equal margins make this
directed graph balanced at every vertex. Decompose it into directed
cycles. On a simple directed cycle, remove the first-kind entries and
add the second-kind entries. This preserves every margin, uses only
allowed entries, and reduces the symmetric difference. Repetition
reaches the second realization. The last statement follows because
all cycle and gap counts, hence their phases and totals, remain fixed.
\end{proof}

Alternating cycles of length greater than four can be essential.
The quota-one supports $\{A,B\},\{B,C\},\{C,A\}$ with column counts
$(1,1,1)$ have two realizations joined by a six-edge alternating cycle,
while their incidence graph has no four-cycle. Thus the permitted
incidences determine the lengths of the exchanges required to connect
selections with fixed margins.

When the sources available to different columns are disjoint and there
are no joint row constraints, the inequalities reduce to the separate
gap capacities in Proposition~\ref{prop:binary-capacity}. For fixed
column counts, integral max-flow decides their simultaneous realization.
Selecting the units $K_C$ and compatible congruent counts is the further
arithmetic problem. For iteration, the choice must also preserve the
intersection connectivity and marked-interval conditions of
Definition~\ref{def:collar}; the selector and relative lift in the main
proof achieve these conditions together.

\section{Formalization and finite verification}\label{app:computation}
Theorems~\ref{thm:even} and~\ref{thm:main} have Lean~4 formalizations in
release \texttt{0.3.0-alleven} of~\cite{TorusLean}, at source revision
\texttt{4900983c4f4b}. The development uses Lean \texttt{4.30.0-rc2}
and mathlib revision \texttt{5450b53e5ddc}. Its vertices are functions
$\mathrm{Fin}\,d\to\mathrm{ZMod}\,m$, and a decomposition assigns every
positive coordinate arc to a colour whose successor is a single cycle.
The final declarations are
\begin{center}
\small
\path{TorusEven.even_modulus_tori_all_dimensions},\\
\path{TorusAll.all_moduli_tori_all_dimensions}.
\end{center}

\subsection{Correspondence with the written proof}
Table~\ref{tab:lean-correspondence} identifies the principal declarations
in the fixed release. The selection and relative-splitting statements use
the global separation of distinguished directions in
Definition~\ref{def:collar}. Its iteration is the declaration
\path{TorusEven.Collar.RelativeCollarState.hamilton_decomposition}.
The initial factorization proof establishes the component parity
required by the induction.

\begin{table}[htbp]
\centering\small
\begin{tabular}{@{}p{0.31\textwidth}p{0.65\textwidth}@{}}
\toprule
Written result & Lean declaration\\
\midrule
Proposition~\ref{prop:pinned-selection} &
\path{TorusEven.Collar.Incidence.pinnedSelection_iff}\\[3pt]
Theorem~\ref{thm:closure} (one split) &
\path{TorusEven.Collar.RelativeCollarState.exists_split}\\[3pt]
Lemma~\ref{lem:star}, Appendix~\ref{app:star} &
\path{TorusEven.Entry.Star.return_hamilton}\\[3pt]
Theorem~\ref{thm:even-dim} &
\path{TorusEven.even_degree_collar}\\[3pt]
Theorem~\ref{thm:oddconstruction} &
\path{TorusEven.even_odd_degree}\\[3pt]
Theorem~\ref{thm:d5large} &
\path{TorusEven.d5_even_large}\\[3pt]
$D_3(m)$, even $m\ge4$ &
\path{TorusEven.d3_even} (Route-E construction)\\[3pt]
$D_5(4)$ &
\path{TorusEven.d5_even_four} (Route-E construction)\\
\bottomrule
\end{tabular}
\caption{Correspondence recorded in the fixed release. The two Route-E
constructions establish the stated existence results by different
explicit witnesses.}\label{tab:lean-correspondence}
\end{table}

The anchored cyclic-star return calculation, including the two cases of
Appendix~\ref{app:star}, is covered by the listed declaration. The additive
case of Theorem~\ref{thm:onegap} is formalized as
\path{TorusEven.Collar.oneGap}. The ordinary three-dimensional existence
result also has the separate Route-E proof used in the final assembly.
The final assembly uses the Route-E witness for $D_5(4)$.
Section~\ref{sec:d5} gives a separate four-row construction, reproduced
by the supplementary calculation below.

The all-moduli declaration combines the even theorem with
\begin{center}\small
\path{RoundComposite.Concrete.odd_modulus_tori_all_dimensions_v75}.
\end{center}
Section~\ref{sec:odd} gives a shorter written proof with an integer
selector. That selector, the permutation-valued generality of
Theorem~\ref{thm:onegap}, the exact incidence-component classifications,
and the additional boundary, allocation, and transfer results are
proved here separately from the cited formalization.

\subsection{Logical dependencies}
Table~\ref{tab:formal-audit} uses three standard axioms:
\path{propext}, \path{Classical.choice}, and \path{Quot.sound}.
Its last column counts the additional axioms generated by finite
\texttt{native\_decide} evaluations. These evaluations use the native
compiler and evaluation machinery as part of the trusted basis~\cite{LeanReference}.

\begin{table}[htbp]
\centering\small
\begin{tabular}{@{}p{0.54\textwidth}p{0.22\textwidth}r@{}}
\toprule
Formal result & Logical axioms & Native axioms\\
\midrule
$D_5(m)$, even $m\ge6$ & standard & 0\\
$D_d(m)$, even $d\ge2$, even $m\ge4$ & standard & 0\\
$D_d(m)$, odd $d\ge7$, even $m\ge4$ & standard & 0\\
Theorem~\ref{thm:even} & standard & 18\\
Theorem~\ref{thm:main} & standard & 106\\
\bottomrule
\end{tabular}
\caption{Logical dependencies recorded by \texttt{\#print axioms} in
release \texttt{0.3.0-alleven}.}\label{tab:formal-audit}
\end{table}

The 18 finite dependencies of the even theorem comprise four checks for
the Route-E witness of $D_5(4)$ and fourteen first-return checks at $m=6,8$
for the Route-E three-dimensional construction. The earlier odd proof
contributes 88 further dependencies. The anchored cyclic-star proof,
the general even-dimensional and odd-dimensional constructions, and
the five-dimensional construction for $m\ge6$ use the standard axioms alone.
The linked axiom ledger gives the declaration names, and the fixed
release contains the corresponding
\href{https://github.com/aria1th/Torus-Hamilton-Decomposition-Program/blob/4900983c4f4be61ed90b008c6063cddd6c91d4e2/evidence/lean_audit_20260911/final-control-check/FinalAudit.lean}{audit source}
and
\href{https://github.com/aria1th/Torus-Hamilton-Decomposition-Program/blob/4900983c4f4be61ed90b008c6063cddd6c91d4e2/evidence/lean_audit_20260911/final-control-check/axioms.log}{raw output}.

\subsection{Verification of the displayed finite data}
The supplementary file \path{anc/verify_finite.py} reconstructs the
printed formulas using Python~3.10 or later and its standard library.
It checks the two complement tables; the maps $j_i$, $N_i$, and the
cyclic lists in Lemma~\ref{lem:d5-endpoint}; and the first-hit calculations
in Appendix~\ref{app:star}. For the five-dimensional construction it
also checks the layer bijections, intermediate cosets, final Hamilton
returns, and the four identities in \eqref{eq:d5four-iterate}.

The command
\begin{center}\small
\texttt{python3 anc/verify\_finite.py}
\end{center}
checks the return tables and planar lists for even $4\le m\le120$,
the full five-dimensional construction for even $4\le m\le14$, and the
binary-cube enumeration cited in the introduction. Command-line options
set the two modulus ranges. The script reproduces these finite instances;
the proofs in the main text and Appendix~\ref{app:star} cover all
admissible moduli.

\section*{Acknowledgements}
I thank Joonkyung Lee for his guidance and for encouraging conversations
during the development of this work.

\section*{Generative-AI disclosure}
Generative AI contributed extensively to the construction of proofs,
finite calculations, Lean formalization, and exposition. The odd-modulus
work used GPT-5.5 Pro and GPT-5.5 Codex; the even-modulus work used
GPT-6-Pro, Claude Fable~5.1, and GPT~6 Astra (max). AI was also used to
compare research notes, develop the auxiliary return and selection
arguments, and revise the manuscript. The earlier version~\cite{ParkVtwo}
and the contribution record in~\cite{TorusLean} describe the division of
work. The author chose the research direction and proof strategy,
reviewed the mathematical development, and is solely responsible for
the results, their attribution, and the final text.

\end{document}